\documentclass[11pt,reqno]{amsart}

\usepackage{aliascnt,amsmath,amssymb,amsthm,geometry,shuffle,xcolor}
\usepackage[abbrev]{amsrefs}

\definecolor{mylinkcolor}{RGB}{16,156,81}
\definecolor{mycitecolor}{RGB}{20,80,140}

\usepackage{hyperref}
\usepackage[nameinlink]{cleveref}
\hypersetup{
  setpagesize=false,
  bookmarksnumbered=true,
  bookmarksopen=true,
  colorlinks=true,
  linkcolor=mylinkcolor,
  citecolor=mycitecolor,
  urlcolor=mycitecolor
}
\numberwithin{equation}{section}

\newtheorem{thm}{Theorem}[section]
\crefname{thm}{Theorem}{Theorems}

\newaliascnt{prop}{thm}
\newtheorem{prop}[prop]{Proposition}
\aliascntresetthe{prop}
\crefname{prop}{Proposition}{Propositions}

\newaliascnt{cor}{thm}
\newtheorem{cor}[cor]{Corollary}
\aliascntresetthe{cor}
\crefname{cor}{Corollary}{Corollaries}

\newaliascnt{lem}{thm}
\newtheorem{lem}[lem]{Lemma}
\aliascntresetthe{lem}
\crefname{lem}{Lemma}{Lemmas}

\newaliascnt{conj}{thm}

\aliascntresetthe{conj}
\crefname{conj}{Conjecture}{Conjectures}

\newtheorem{mainthm}{Theorem}

\crefname{mainthm}{Theorem}{Theorems}

\newtheorem{mainthmB}{Theorem}

\crefname{mainthmB}{Theorem}{Theorems}

\newtheorem{mainthmC}{Theorem}

\crefname{mainthmC}{Theorem}{Theorems}

\newtheorem{mainthmD}{Theorem}

\crefname{mainthmD}{Theorem}{Theorems}

\newtheorem{mainthmE}{Theorem}

\crefname{mainthmE}{Theorem}{Theorems}

\newtheorem{mainconjF}{Conjecture}

\crefname{mainconjF}{Conjecture}{Conjectures}

\newtheorem{mainthmG}{Theorem}

\crefname{mainthmG}{Theorem}{Theorems}

\theoremstyle{definition}
\newaliascnt{rem}{thm}
\newtheorem{rem}[rem]{Remark}
\aliascntresetthe{rem}
\crefname{rem}{Remark}{Remarks}
\theoremstyle{plain}

\newcommand{\ARI}{\operatorname{ARI}}
\newcommand{\BARI}{\operatorname{BARI}}
\newcommand{\QQ}{\mathbb Q}
\newcommand{\swapop}{\operatorname{swap}}
\newcommand{\pushop}{\operatorname{push}}
\newcommand{\uri}{\operatorname{uri}}
\newcommand{\preuri}{\operatorname{preuri}}
\newcommand{\preari}{\operatorname{preari}}
\newcommand{\ari}{\operatorname{ari}}
\newcommand{\urit}{\operatorname{urit}}
\newcommand{\arit}{\operatorname{arit}}
\newcommand{\amit}{\operatorname{amit}}
\newcommand{\anit}{\operatorname{anit}}
\newcommand{\axit}{\operatorname{axit}}
\newcommand{\muop}{\operatorname{mu}}
\newcommand{\gaxit}{\operatorname{gaxit}}
\newcommand{\ganit}{\operatorname{ganit}}
\newcommand{\pic}{\operatorname{pic}}
\newcommand{\poc}{\operatorname{poc}}
\newcommand{\picy}{\operatorname{pic}^{Y}}
\newcommand{\invmu}{\operatorname{invmu}}
\newcommand{\one}{\mathbf 1}
\newcommand{\cL}{\mathcal L}
\newcommand{\BARIloc}{\BARI^{\mathrm{loc}}}
\newcommand{\BARIswapil}{\BARI_{\swapop,\underline{\mathrm{il}}}}
\newcommand{\bi}[2]{\genfrac{(}{)}{0pt}{}{#1}{#2}}
\newcommand{\checkD}{\check{D}}
\newcommand{\checkW}{\check{W}}
\newcommand{\checkdelta}{\check{\delta}}
\newcommand{\EX}{E_X}
\newcommand{\EY}{E_Y}
\DeclareRobustCommand{\slfrak}{\mathfrak{sl}}

\title{An \texorpdfstring{$\slfrak_2$}{sl2}-action on the uri Lie algebra}
\author{Henrik Bachmann}
\address{Graduate School of Mathematics, Nagoya University, Nagoya, Japan.}
\email{henrik.bachmann@math.nagoya-u.ac.jp}
\date{\today}
\subjclass[2020]{Primary 11M32; Secondary 11F11, 11F67, 17B01}
\keywords{bimoulds, uri bracket, formal multiple Eisenstein series, $\mathfrak{sl}_2$-action}

\begin{document}

\begin{abstract}
We construct an $\mathfrak{sl}_2$-action by derivations on the Lie algebra $\mathfrak B$ of swap invariant alternil bimoulds with the uri bracket of K\"uhn and Schneps.  This Lie algebra plays the role for formal multiple Eisenstein series which Racinet's double shuffle Lie algebra $\mathfrak{dm}_0$ plays for multiple zeta values.  The kernel $\mathfrak m$ of the lowering operator is a Lie subalgebra, $\mathfrak B$ is the direct sum of its iterates under the raising operator, and this gives Rankin--Cohen type operations $\uri_n$ on $\mathfrak m$.  We define a Lie subalgebra $\mathfrak d$ of $\mathfrak B$ and show that it is isomorphic to $\mathfrak{dm}_0$ extended by an additional element in weight one.
\end{abstract}

\maketitle

\section{Introduction}

Multiple Eisenstein series were introduced by Gangl, Kaneko and Zagier in \cite{GKZ} as hybrids of classical Eisenstein series and multiple zeta values.  They are holomorphic functions on the upper half-plane, whose Fourier expansions have multiple zeta values as their constant terms, and in depth one they are the classical Eisenstein series.  In particular, they span an algebra containing the algebra $\widetilde{\mathcal M}$ of quasimodular forms, which is an example of an $\mathfrak{sl}_2$-algebra, i.e.\ there are three derivations $W$, $D$ and $\delta$ on $\widetilde{\mathcal M}$, among them the derivative $D=q\frac{d}{dq}$, satisfying the commutator relations of $\mathfrak{sl}_2$ (see \cite{Za}).

The algebra $\mathcal G^{\!f}$ of formal multiple Eisenstein series was introduced by the author and van Ittersum in \cite{BIM}.  It is spanned by formal symbols which are swap invariant and whose product is the bi-stuffle product, and conjecturally these are exactly the relations satisfied by (regularized) multiple Eisenstein series.  The algebra $\mathcal G^{\!f}$ can also be seen as the formal weight-graded version of the algebra of multiple $q$-zeta values in \cite{BK}, since conjecturally its relations are exactly the top-weight parts of the relations among these $q$-series.  Hence the dimension conjectures in \cite{BK} for the associated graded spaces are the same as those for (formal) multiple Eisenstein series.  Realizations of $\mathcal G^{\!f}$ give back the original objects.  In \cite{BB1} the combinatorial multiple Eisenstein series realize $\mathcal G^{\!f}$ in $\QQ[[q]]$, and in \cite{BKM} the bi-multiple Eisenstein series recover the multiple Eisenstein series themselves.  The formal quasimodular subalgebra of $\mathcal G^{\!f}$ is isomorphic to $\widetilde{\mathcal M}$, and in \cite{BIM} the classical $\mathfrak{sl}_2$-action on $\widetilde{\mathcal M}$ extends to an $\mathfrak{sl}_2$-action by derivations on all of $\mathcal G^{\!f}$.

Racinet's double shuffle Lie algebra $\mathfrak{dm}_0$ encodes the double shuffle relations of multiple zeta values.  Write $\mathcal U(\mathfrak g)^\vee$ for the graded dual of the universal enveloping algebra of a weight-graded Lie algebra $\mathfrak g$.  The expected decomposition for the algebra $\mathcal Z$ of multiple zeta values is
\[
 \mathcal Z
 \overset{?}{\cong}
 \QQ[\zeta(2)]\otimes\mathcal U(\mathfrak{dm}_0)^\vee.
\]
For formal multiple zeta values, the corresponding statement is a theorem of Racinet \cite{Rac} (see \cite{Bu1}*{Corollary~B.32}).  The expected analogue for formal multiple Eisenstein series, proposed in \cite{BK}, is
\begin{equation}\label{eq:Gf-conjecture}
 \mathcal G^{\!f}
 \overset{?}{\cong}
 \widetilde{\mathcal M}\otimes\mathcal U(\mathfrak B)^\vee.
\end{equation}
Here $\mathfrak B=\bigoplus_{k\geq1}\mathfrak B_k$ is the finite-depth polynomial part of the space $\BARIswapil$ of power-series bimoulds which are alternil, swap invariant and even in depth one.  The uri bracket was proposed by K\"uhn and Schneps and studied by Burmester in \cite{Bu1} (see also \cite{BKS}).  In \cite{BKa}*{Theorem~A(ii)}, the author and Kawamura proved that $\mathfrak B$ is a weight-graded Lie algebra with the uri bracket.

Both the left-hand side of \eqref{eq:Gf-conjecture} and its factor $\widetilde{\mathcal M}$ carry $\mathfrak{sl}_2$-actions.  This paper constructs the corresponding action by derivations on $\mathfrak B$. We denote its operators by $\checkdelta,\checkW,\checkD$, and those on $\mathcal G^{\!f}$ by $D,W,\delta$, so that $\checkdelta$ plays the role of the raising derivation $D$ on quasimodular forms and $\checkD$ the role of the lowering derivation $\delta$.  Our main result is the following.

\begin{mainthm}\label{thm:main}
The Lie algebra $(\mathfrak B,\uri)$ is an $\mathfrak{sl}_2$-Lie algebra.  This means that there exist three derivations $\checkdelta$, $\checkW$ and $\checkD$, defined in \Cref{sec:three-derivations}, of weights $2$, $0$ and $-2$, respectively, which satisfy
\begin{equation}\label{eq:sl2-relations}
 [\checkW,\checkdelta]=2\checkdelta,
 \qquad
 [\checkW,\checkD]=-2\checkD,
 \qquad
 [\checkD,\checkdelta]=\checkW.
\end{equation}
\end{mainthm}

Before restricting to $\mathfrak B$, the operators have simple descriptions.  In depth $r$, the map $\checkW$ is the Euler operator plus $r$, while $\checkD=\sum_i\partial_{X_i}\partial_{Y_i}$.  The operator $\checkdelta$ is obtained by transporting multiplication by $\sum_iX_iY_i$ from ari to uri and adding an inner uri derivation.  Their derivation and commutator relations on $(\BARIloc,\uri)$ follow from the blockwise identities in \cref{prop:ambient-triple}.

The main point is that $\checkdelta$ preserves the defining conditions of $\mathfrak B$, in particular swap invariance.  Comparison with the derivation of \cite{BIM} and with preuri reduces this to the $\mathfrak E$-push invariance proved in \cite{Kaw2}.  The same comparison shows that uri is the commutator of preuri on polynomial bimoulds and gives a new expression for the derivation of \cite{BIM} in terms of ganit and preuri (see \cref{prop:preuri-comparison,prop:delta-comparison}).

In \Cref{sec:kernel} we study the kernel of $\checkD$ on $\mathfrak B$. For quasimodular forms, the kernel of the weight-lowering derivation is exactly the space of modular forms, and every quasimodular form is a linear combination of derivatives of modular forms and of derivatives of $G_2$ (see \cite{Za}).  On $\mathfrak B$ we put $\mathfrak m=\ker(\checkD)$.  Our second main result is the analogue for $\mathfrak B$ of the above description of the quasimodular forms.

\begin{mainthmB}\label{thm:B}
\begin{enumerate}
\item[(i)] The space $\mathfrak m$ is a weight-graded Lie subalgebra of
$(\mathfrak B,\uri)$.
\item[(ii)] The operator $\checkdelta$ is injective, and $\checkD$ is
surjective.
\item[(iii)] We have
\[
 \mathfrak B=\bigoplus_{m\geq0}\checkdelta^m(\mathfrak m)\,,
\]
and in particular $\dim_\QQ\mathfrak m_k=\dim_\QQ\mathfrak B_k-\dim_\QQ\mathfrak B_{k-2}$.
\end{enumerate}
\end{mainthmB}

The classical Rankin--Cohen brackets combine derivatives of two modular forms to produce another modular form.  Inside the quasimodular forms, the modular forms are precisely $\ker(\delta)$, so the brackets combine terms $D^r(f)D^s(g)$ with $f,g\in\ker(\delta)$ into another element of $\ker(\delta)$ (see \cite{Za}).  This construction only uses the $\mathfrak{sl}_2$-relations and the derivation property.  Applying it to $\mathfrak m=\ker(\checkD)$, with $\checkdelta$ as the weight-raising operator and uri in place of the product, gives our third main result.

\begin{mainthmC}\label{thm:C}
For $f\in\mathfrak m_k$, $g\in\mathfrak m_l$ and $n\geq0$, the expression
\begin{equation}\label{eq:uri-n}
 \uri_n(f,g)
 =\sum_{\substack{r,s\geq0\\r+s=n}}
 (-1)^r\binom{n+k-1}{s}\binom{n+l-1}{r}
 \uri\bigl(\checkdelta^r(f),\checkdelta^s(g)\bigr)
\end{equation}
belongs to $\mathfrak m_{k+l+2n}$ and satisfies
\[
 \uri_n(g,f)=(-1)^{n+1}\uri_n(f,g)\,.
\]
\end{mainthmC}

More precisely, \cref{prop:uri-n-projection} shows that, up to an explicit nonzero scalar, $\uri_n(f,g)$ is the projection to $\mathfrak m$ of every uri bracket $\uri(\checkdelta^r f,\checkdelta^s g)$ with $r+s=n$.  In particular, $\uri_0=\uri$.

Put $\xi_1=(1,0,\ldots)\in\mathfrak B_1$.  For $r\geq0$, this element gives the inner derivation
\[
 \partial_{2r+1}
 =\operatorname{ad}_{\uri}(\checkdelta^r\xi_1):
 \mathfrak B_k\longrightarrow\mathfrak B_{k+2r+1}.
\]
We denote the kernel of $\partial_1=\operatorname{ad}_{\uri}(\xi_1)$ by $\mathfrak d=\ker(\partial_1)$.  These derivations have the following elegant compatibility with the $\mathfrak{sl}_2$-action.

\begin{mainthmD}\label{thm:D}
\begin{enumerate}
\item[(i)] Each $\partial_{2r+1}$ is a derivation of $\mathfrak B$ of
weight $2r+1$, and for $r\geq0$ we have
\[
 [\checkW,\partial_{2r+1}]=(2r+1)\partial_{2r+1},
 \qquad
 [\checkD,\partial_{2r+1}]=r^2\partial_{2r-1},
 \qquad
 [\checkdelta,\partial_{2r+1}]=\partial_{2r+3},
\]
where $\partial_{-1}=0$.
\item[(ii)] The space $\mathfrak d$ is a weight-graded uri Lie subalgebra
of $\mathfrak B$.
\end{enumerate}
\end{mainthmD}
We also prove that for $A\in\mathfrak m$ the component in $\mathfrak m$ of $\partial_{2r+1}(A)$ with respect to the decomposition in \cref{thm:B}(iii) is, up to an explicit nonzero scalar, the element $\uri_r(\xi_1,A)$ (see \eqref{eq:partial-uri-n}).  Hence the operations $\uri_r(\xi_1,-)$ are exactly the components in $\mathfrak m$ of the odd derivations acting on $\mathfrak m$.  We will also prove the chain of Lie algebras $\mathfrak d\subseteq\mathfrak m\subseteq\mathfrak B$.

In \Cref{sec:dm0} we recall the definition of $\mathfrak{dm}_0$, equipped with the Ihara bracket $\{\ ,\ \}_{\mathrm{Ih}}$.  Adjoin a formal central generator $\sigma_1$ of weight one and write $\mathfrak{dm}_0^+=\QQ\sigma_1\oplus\mathfrak{dm}_0$.  In \eqref{eq:dm-to-m} we construct an explicit map $\iota:\mathfrak{dm}_0^+\to\mathfrak d$ with $\iota(\sigma_1)=\xi_1$.

\begin{mainthmE}\label{thm:E}
The map $\iota$ is a Lie algebra isomorphism
\[
 \iota:
 \bigl(\mathfrak{dm}_0^+,\{\ ,\ \}_{\mathrm{Ih}}\bigr)
 \overset{\sim}{\longrightarrow}
 \bigl(\mathfrak d,\uri\bigr).
\]
\end{mainthmE}

The proof uses an explicit formula for $\partial_1$ and a leading-term argument which does not involve the defining conditions of $\mathfrak B$.  In particular, $\mathfrak d$ consists exactly of the elements of $\mathfrak B$ without mixed $X,Y$-monomials, and the restriction of $\partial_{2r+1}$ to $\mathfrak m$ is injective for every $r\geq1$.  We also show that $\ker(\partial_{2r+1})$ is one-dimensional for every $r\geq1$ (\cref{thm:odd-kernels}).  A similar statement to \cref{thm:E} is proven independently in \cite{BKS}.

Conjecturally, $\mathfrak{dm}_0$ is generated by one element in each odd weight $k\geq3$.  Using the canonical genus-zero elements $\sigma_k$ constructed in \cite{DDDHKSSV},\footnote{They are denoted by $g_k$ there.} together with the formal element $\sigma_1$, we define $\xi_k=\iota(\sigma_k)$ for odd $k\geq1$.  The Eisenstein relations $\uri_0(\xi_1,\xi_k)=0$ place all the $\xi_k$ in $\mathfrak d$.  Formula \eqref{eq:partial-uri-n} shows how the higher operations $\uri_r(\xi_1,-)$ measure the action of the remaining elements in the $\checkdelta$-orbit of $\xi_1$. Ordinary uri brackets of the $\xi_k$ remain in the image of $\iota$, while the higher operations need not preserve this image.  For example, $\uri_1(\xi_1,\xi_1)\in\mathfrak m_4$ has mixed $X,Y$-monomials. We write $\langle\xi_k\mid k\geq1\text{ odd}\rangle_{\uri_\bullet}$ for the smallest graded subspace of $\mathfrak m$ which contains these elements and is closed under all the operations $\uri_n$.

\begin{mainconjF}\label{conj:uri-n-generation}
\begin{enumerate}
\item[(i)] The elements $\xi_k$ generate $\mathfrak m$ under the
operations $\uri_n$, i.e.\
\[
 \mathfrak m
 =\left\langle\xi_k\ \middle|\ k\geq1\text{ odd}\right\rangle_{\uri_\bullet}.
\]
\item[(ii)] The elements $\checkdelta^r(\xi_k)$ with $r\geq0$ and odd
$k\geq1$ generate the Lie algebra $(\mathfrak B,\uri)$.
\end{enumerate}
\end{mainconjF}

Part (ii) is due to K\"uhn, who verified it through weight thirteen (see also \cite{BKS}).  It is also stated in \cite{Bu1}*{Conjecture~1.22(ii)}.  The two parts are equivalent by the change of basis in \cref{lem:uri-n-generation}.

In \eqref{eq:C-fg} we define, for period polynomials $f,g\in W_k^\pm$, an explicit binary $\uri_n$-combination $\mathcal C_{f,g}\in\mathfrak m_k$. Similar elements, using the polynomials $P_{f,g}$ studied in \cite{Co}, together with the associated depth-graded Lie algebra, are considered in \cite{BKS}.   To state our final result, let $\mathcal F_k$ be the span of the weight-$k$ expressions obtained by iterating the operations $\uri_n$ on the $\xi_j$ and containing at least four occurrences of these generators.

\begin{mainthmG}\label{thm:G}
\begin{enumerate}
\item[(i)] For even $k\geq4$ and $f,g\in W_k^\pm$, the element
$\mathcal C_{f,g}$ vanishes in depths at most three.
\item[(ii)] For the period polynomials $p^\pm_\Delta\in W_{12}^\pm$ of
the Ramanujan $\Delta$ function, there are explicit correction terms $Q^\pm\in\mathcal F_{12}$, given in the \hyperref[sec:weight-twelve-corrections]{Appendix}, such that
\[
 \mathcal C_{p^+_\Delta,p^+_\Delta}+Q^+=0,
 \qquad
 \mathcal C_{p^-_\Delta,p^-_\Delta}+Q^-=0.
\]
\end{enumerate}
\end{mainthmG}

\noindent\textbf{Acknowledgements.}
The author thanks Ulf K\"uhn for helpful comments on an earlier version of this paper and Hanamichi Kawamura for his help with understanding mould theory over the past year.
This project was partially supported by JSPS KAKENHI Grant 26K22254.

\section{The three derivations}\label{sec:three-derivations}

\subsection{Bimoulds and the uri bracket}

We use \'{E}calle's bimould and flexion formalism \cites{Ec1,Ec2}, with the conventions of \cites{Sc,Kaw1} and the coefficient normalization of \cite{BIM}.

A bimould is a family
\[
 A=(A_r)_{r\geq0},
 \qquad
 A_r\in\QQ[[X_1,Y_1,\ldots,X_r,Y_r]].
\]
We write $\BARI$ for the bimoulds with $A_0=0$.  When using Kawamura's notation, this space equipped with the ari bracket below is denoted by $\ARI$.  A bimould is normalized if $A_0=1$, and $\one=(1,0,0,\ldots)$ denotes the normalized unit.  We write $\BARI^{\mathrm{pol}}$ for the polynomial bimoulds with only finitely many nonzero components.

Every polynomial bimould has an expansion
\begin{equation}\label{eq:coefficient-expansion}
 A_r\bi{X_1,\ldots,X_r}{Y_1,\ldots,Y_r}
 =
 \sum_{\substack{k_1,\ldots,k_r\geq1\\m_1,\ldots,m_r\geq0}}
 a_A\bi{k_1,\ldots,k_r}{m_1,\ldots,m_r}
 \prod_{i=1}^r\frac{X_i^{k_i-1}Y_i^{m_i}}{m_i!}.
\end{equation}
Let $y_{k,m}$ be a letter of weight $k+m$.  On words in these letters, the shuffle product is defined by $1\shuffle w=w\shuffle1=w$ and
\[
 y_{k,m}u\shuffle y_{l,n}v
 =y_{k,m}(u\shuffle y_{l,n}v)
  +y_{l,n}(y_{k,m}u\shuffle v).
\]
The bi-stuffle product is defined by $1*w=w*1=w$ and
\begin{align*}
 y_{k,m}u*y_{l,n}v
 ={}&y_{k,m}(u*y_{l,n}v)
 +y_{l,n}(y_{k,m}u*v)
 +y_{k+l,m+n}(u*v).
\end{align*}
For a bimould $A$, the coefficient map sends $y_{k_1,m_1}\cdots y_{k_r,m_r}$ to the coefficient in \eqref{eq:coefficient-expansion}.  The bimould $A$ is \emph{alternal} if this map vanishes on $u\shuffle v$, and \emph{alternil} if it vanishes on $u*v$, for every pair of nonempty words $u,v$.  For bimoulds with poles, let $w_i=\bi{X_i}{Y_i}$ and take words $u=w_1\cdots w_p$ and $v=w_{p+1}\cdots w_{p+q}$, put
\[
A(u\shuffle v)
=\sum_{\sigma\in\operatorname{Sh}(p,q)}
 A_{p+q}(w_{\sigma(1)},\ldots,w_{\sigma(p+q)}).
\]
Here $\operatorname{Sh}(p,q)$ consists of the permutations $\sigma\in S_{p+q}$ such that
\[
 \sigma^{-1}(1)<\cdots<\sigma^{-1}(p),\qquad
 \sigma^{-1}(p+1)<\cdots<\sigma^{-1}(p+q).
\]
Alternality then means $A(u\shuffle v)=0$. Alternility for bimoulds with poles is defined by the corresponding flexed bi-stuffle identities of \cite{BKa}*{Section~2}.

A homogeneous bimould has weight $k$ if its coefficient in \eqref{eq:coefficient-expansion} can be nonzero only when
\begin{equation*}
 \sum_{i=1}^r(k_i+m_i)=k.
\end{equation*}
The swap and push operators are
\begin{align*}
 \swapop(A)_r\bi{X_1,\ldots,X_r}{Y_1,\ldots,Y_r}
 &={}
 A_r\bi{Y_1+\cdots+Y_r,\ldots,Y_1+Y_2,Y_1}
 {X_r,X_{r-1}-X_r,\ldots,X_1-X_2},
 \\
 \pushop(A)_r\bi{X_1,\ldots,X_r}{Y_1,\ldots,Y_r}
 &={}
 A_r\bi{-X_r,X_1-X_r,\ldots,X_{r-1}-X_r}
 {-Y_1-\cdots-Y_r,Y_1,\ldots,Y_{r-1}}.
\end{align*}
As in \cite{BKa}*{Section~2}, we write $\BARIswapil$ for the power-series bimoulds which are alternil, swap invariant and even in depth one. The finite-depth polynomial part of $\BARIswapil$ is therefore
\begin{equation*}
 \mathfrak B
 =\left\{
 A\in\BARI^{\mathrm{pol}}
 \ \middle|\
 A\text{ is alternil},\quad
 \swapop(A)=A,\quad
 A_1\bi{-X_1}{-Y_1}=A_1\bi{X_1}{Y_1}
 \right\}.
\end{equation*}
Its homogeneous weight-$k$ part is denoted by $\mathfrak B_k$, and we set $\mathfrak B_k=0$ for $k\leq0$.  Notice that a homogeneous depth-$r$ polynomial of ordinary degree $d$ has weight $r+d$.

The flexion operators produce poles along the hyperplanes $X_i=0$ and $X_i=X_j$.  We therefore put
\begin{equation*}
 \cL_r
 =\QQ[[X_1,Y_1,\ldots,X_r,Y_r]]
 \left[
 \frac1{X_i},\frac1{X_i-X_j}
 \ \middle|\ 1\leq i\ne j\leq r
 \right]
\end{equation*}
and let $\BARIloc$ consist of the bimoulds with $A_0=0$ and $A_r\in\cL_r$.

The concatenation product is
\begin{equation*}
 \muop(A,B)(w_1\cdots w_r)
 =\sum_{i=0}^rA(w_1\cdots w_i)B(w_{i+1}\cdots w_r).
\end{equation*}
Its unit is $\one$, and $\invmu(A)$ denotes the inverse of a normalized bimould.  For $s\geq3$, we write
\[
 \muop(M_1,\ldots,M_s)
 =\muop\bigl(\muop(M_1,\ldots,M_{s-1}),M_s\bigr).
\]
Given a bimould $M$, a linear form $t$ and $p\leq q$, put
\[
 M^{(t)}\langle p,q\rangle
 =M_{q-p+1}\bi{X_p-t,\ldots,X_q-t}{Y_p,\ldots,Y_q},
 \qquad
 M^{(t)}\langle p,p-1\rangle=1,
\]
and write $Y_{p:q}=Y_p+\cdots+Y_q$.  For normalized bimoulds $B,C$, define
\begin{align*}
 \gaxit_{B,C}(M)(w_1\cdots w_r)
 ={}&\sum_{s\geq1}
 \sum_{\substack{0=q_0<p_1\leq q_1<\cdots<p_s\leq q_s=r}}
 M(\widehat w_{p_1}\cdots\widehat w_{p_s})\notag\\
 &\quad\cdot\prod_{i=1}^s
 B^{(X_{p_i})}\langle q_{i-1}+1,p_i-1\rangle
 C^{(X_{p_i})}\langle p_i+1,q_i\rangle,
\end{align*}
where
\begin{equation*}
 \widehat w_{p_i}=\bi{X_{p_i}}{Y_{q_{i-1}+1:q_i}},
 \qquad
 \gaxit_{B,C}(M)(\varnothing)=M(\varnothing).
\end{equation*}
This is the standard gaxit action of \cite{Kaw1}*{Definition~3.7}.  It is multiplicative for $\muop$ \cite{Kaw1}*{Proposition~3.12}.

Put $\ganit_C=\gaxit_{\one,C}$.  The GAXI composition law makes this operator invertible for every normalized $C$ \cite{Kaw1}*{Proposition~3.13}.  Over the dual numbers, define
\begin{equation*}
 \gaxit_{\one+\varepsilon B,\one+\varepsilon C}
 =\operatorname{id}+\varepsilon\axit(B,C),
 \qquad \varepsilon^2=0,
\end{equation*}
and set
\begin{equation*}
 \amit(B)=\axit(B,0),
 \qquad
 \anit(B)=\axit(0,B),
 \qquad
 \arit(B)=\axit(B,-B).
\end{equation*}
Since gaxit is multiplicative, every $\axit(B,C)$ is a derivation for $\muop$. Finally,
\begin{align*}
 \preari(A,B)&=\arit(B)(A)+\muop(A,B),
 \\
 \ari(A,B)&=\preari(A,B)-\preari(B,A).
\end{align*}
By the gaxit composition law, ari is a Lie bracket \cite{Sc}*{Proposition~2.2.2}.

The normalized bimoulds pic, poc and $\operatorname{pic}^{Y}$ are given in positive depth by
\begin{equation*}
 \pic_r=\frac1{X_1\cdots X_r},
 \qquad
 \poc_r=-\frac1{X_1(X_1-X_2)\cdots(X_{r-1}-X_r)},
 \qquad
 (\picy)_r=\frac1{Y_1\cdots Y_r}.
\end{equation*}
Applying the flexion-unit inverse formula in \cite{Kaw1}*{Proposition~5.3} gives
\begin{equation*}
 T:=\ganit_{\pic},
 \qquad
 T^{-1}=\ganit_{\poc},
\end{equation*}
and $T$ identifies alternal and alternil bimoulds \cite{Ko}*{Theorem~3.24} (see also \cite{Bu1}*{Proposition~5.9}).  The uri bracket is the transport of ari by $T$,
\begin{equation*}
 \uri(A,B)=T\bigl(\ari(T^{-1}A,T^{-1}B)\bigr).
\end{equation*}

\subsection{The operators}

Let $E$ be the bimould whose only nonzero components are
\begin{equation}\label{eq:E-definition}
 E_1\bi{X_1}{Y_1}=-\frac{X_1+Y_1}{2},
 \qquad
 E_2\bi{X_1,X_2}{Y_1,Y_2}=\frac14.
\end{equation}
Notice that $E=\delta_{\mathrm{BvI}}(\one)$, where $\delta_{\mathrm{BvI}}$ is the bimould operator of \cite{BIM} recalled in \eqref{eq:BvI-definition}.  In \cref{prop:delta-comparison} we show that $\checkdelta$ and $\delta_{\mathrm{BvI}}$ differ by $\preuri(E,-)$.

For $A=(A_r)_{r\geq0}\in\BARIloc$, define
\begin{align}
 (QA)_r
 &=\left(\sum_{i=1}^rX_iY_i\right)A_r,
 \notag\\
 (\checkW A)_r
 &=\left(r+\sum_{i=1}^r
 X_i\partial_{X_i}+Y_i\partial_{Y_i}\right)A_r,
 \notag\\
 (\checkD A)_r
 &=\sum_{i=1}^r\partial_{X_i}\partial_{Y_i}A_r,
 \notag\\
 \checkdelta(A)
 &=\ganit_{\pic}\bigl(Q(\ganit_{\poc}(A))\bigr)
   +\uri(A,E).
 \label{eq:delta-definition}
\end{align}

\begin{prop}\label{prop:ambient-triple}
The maps $\checkW$, $\checkD$ and $\checkdelta$ are derivations of $(\BARIloc,\uri)$ and satisfy \eqref{eq:sl2-relations}.  Moreover, $\checkW$ and $\checkD$ preserve $\mathfrak B$ and have weights $0$ and $-2$, respectively.
\end{prop}

\begin{proof}
To check the derivation properties on the ari side, consider a block $[a,b]$ with marked index $p$, and put
\[
 x_0=X_p,
 \qquad y_0=\sum_{j=a}^bY_j,
 \qquad x_j=X_j-X_p,
 \qquad y_j=Y_j\quad(j\ne p).
\]
The chain rule and a direct calculation give
\begin{align}
 \sum_{j=a}^b
 \left(X_j\partial_{X_j}+Y_j\partial_{Y_j}\right)
 &=x_0\partial_{x_0}+y_0\partial_{y_0}
  +\sum_{\substack{a\leq j\leq b\\j\ne p}}
   \left(x_j\partial_{x_j}+y_j\partial_{y_j}\right),
 \notag\\
 \sum_{j=a}^b\partial_{X_j}\partial_{Y_j}
 &=\partial_{x_0}\partial_{y_0}
  +\sum_{\substack{a\leq j\leq b\\j\ne p}}
   \partial_{x_j}\partial_{y_j},
 \label{eq:D-flexion}\\
 \sum_{j=a}^bX_jY_j
 &=x_0y_0+
 \sum_{\substack{a\leq j\leq b\\j\ne p}}x_jy_j.
 \notag
\end{align}
Every term of $\arit(B)(A)$ has exactly one nontrivial block $[a,b]$. The marked pair $(x_0,y_0)$ and the unchanged pairs outside this block are arguments of the $A$-factor, while the other pairs in the block are arguments of the $B$-factor.  In each term, the output depth is the sum of the two factor depths.  Therefore
\begin{align*}
 \checkW\arit(B)(A)&=\arit(\checkW B)(A)+\arit(B)(\checkW A),\\
 \checkD\arit(B)(A)&=\arit(\checkD B)(A)+\arit(B)(\checkD A),\\
 Q\arit(B)(A)&=\arit(QB)(A)+\arit(B)(QA).
\end{align*}
For $\muop$ the same identities follow by splitting the variables at the concatenation cut.  Hence $\checkW$, $\checkD$ and $Q$ are derivations of preari and of ari.  Componentwise differentiation gives
\begin{equation}\label{eq:D-Q}
 [\checkD,Q]
 =r+\sum_{i=1}^r
 \left(X_i\partial_{X_i}+Y_i\partial_{Y_i}\right)
 =\checkW.
\end{equation}
Homogeneity also gives
\begin{equation}\label{eq:W-D-Q}
 [\checkW,\checkD]=-2\checkD,
 \qquad
 [\checkW,Q]=2Q.
\end{equation}

Since $\pic$ and $\poc$ have $\checkW$-weight zero, $\checkW\ganit_{\pic}=\ganit_{\pic}\checkW$ and $\checkW\ganit_{\poc}=\ganit_{\poc}\checkW$.  To treat $\checkD$, recall that $\ganit_C=\gaxit_{\one,C}$.  Every marked letter is therefore the first letter of its block.  The pairs $(x_0,y_0)$, one from each block, are the arguments of $A$, while the remaining pairs are the arguments of the corresponding $C$-factors.  Different blocks use disjoint sets of variables.  Formula \eqref{eq:D-flexion} therefore gives
\begin{equation*}
 \checkD\ganit_C(A)=\ganit_C(\checkD A)
 \qquad\text{if }\checkD C=0.
\end{equation*}
Both $\pic$ and $\poc$ depend only on the $X$-variables.  Therefore, $\checkD$ commutes with $\ganit_{\pic}$ and $\ganit_{\poc}$.  Transport from ari to uri gives the corresponding relations for $\ganit_{\pic}Q\ganit_{\poc}$.

The bimould $E$ satisfies
\begin{equation*}
 \checkW E=2E,
 \qquad
 \checkD E=0.
\end{equation*}
Also, $A\mapsto\uri(A,E)$ is an inner derivation.  Since $\checkW$ and $\checkD$ are uri derivations,
\begin{equation}\label{eq:inner-relations}
 [\checkW,\uri(-,E)]=2\uri(-,E),
 \qquad
 [\checkD,\uri(-,E)]=0.
\end{equation}
Equations \eqref{eq:D-Q}, \eqref{eq:W-D-Q}, and \eqref{eq:inner-relations} prove \eqref{eq:sl2-relations}.

The operator $\checkW$ preserves $\mathfrak B$ by homogeneity.  On the coefficient side, $\checkD$ is precomposition with the quasi-shuffle derivation determined on letters by
\[
 d(y_{k,m})=k y_{k+1,m+1}
\]
\cite{BIM}*{Definition~4.8 and Proposition~4.9}. If the coefficient map of $A$ vanishes on products of nonempty words, then its composition with $d$ has the same property because
\[
 d(u*v)=d(u)*v+u*d(v).
\]
This proves that $\checkD$ preserves alternility.  Under swap, the variables become
\[
 U_j=Y_1+\cdots+Y_{r-j+1},
 \qquad
 V_j=X_{r-j+1}-X_{r-j+2},
\]
where $X_{r+1}=0$.  Therefore
\[
 \partial_{X_i}
 =\partial_{V_{r-i+1}}-\partial_{V_{r-i+2}},
 \qquad
 \partial_{Y_i}=\sum_{j=1}^{r-i+1}\partial_{U_j},
\]
with $\partial_{V_{r+1}}=0$.  Summing and telescoping gives
\[
 \sum_{i=1}^r\partial_{X_i}\partial_{Y_i}
 =\sum_{j=1}^r\partial_{U_j}\partial_{V_j}.
\]
It follows that $\checkD$ commutes with swap.  Polynomiality and the depth-one parity are immediate.  The weight shifts of $\checkW$ and $\checkD$ follow from their definitions.
\end{proof}

\section{Comparison with the explicit operators}

Write $X=(X_1,\ldots,X_r)$ and $Y=(Y_1,\ldots,Y_r)$.  For $1\leq i\leq r$, let $p_i(X)$ be the tuple obtained by deleting $X_i$.  Put
\begin{align*}
 c_{i,i+1}(Y_1,\ldots,Y_r)
 &= (Y_1,\ldots,Y_{i-1},Y_i+Y_{i+1},Y_{i+2},\ldots,Y_r)
 &&(i<r),\\
 c_{r,r+1}(Y_1,\ldots,Y_r)
 &= (Y_1,\ldots,Y_{r-1}),\\
 c_{i-1,i}(Y_1,\ldots,Y_r)
 &= (Y_1,\ldots,Y_{i-2},Y_{i-1}+Y_i,Y_{i+1},\ldots,Y_r)
 &&(2\leq i\leq r).
\end{align*}
The contraction operators are
\begin{align*}
 (\varphi_i^+A)_r(X\mid Y)
 &=A_{r-1}\bigl(p_i(X)\mid c_{i,i+1}(Y)\bigr)
 &&(1\leq i\leq r),\\
 (\varphi_i^-A)_r(X\mid Y)
 &=A_{r-1}\bigl(p_i(X)\mid c_{i-1,i}(Y)\bigr)
 &&(2\leq i\leq r),
\end{align*}
and are zero outside the indicated ranges.  Set $(\delta_{\mathrm{BvI}}A)_0=0$.  With $X_{r+1}=0$, define its positive-depth components by
\begin{align}
 (\delta_{\mathrm{BvI}}A)_r
 ={}&(QA)_r
 -\frac12\sum_{i=1}^r
 (X_i-X_{i+1}+Y_i)(\varphi_i^+A)_r\notag\\
 &-\frac12\sum_{i=2}^r
 (X_{i-1}-X_i+Y_i)(\varphi_i^-A)_r\notag\\
 &+\frac14\sum_{i=1}^{r-1}((\varphi_i^+)^2A)_r
 -\frac14\sum_{i=2}^{r-1}((\varphi_i^-)^2A)_r.
 \label{eq:BvI-definition}
\end{align}
This is the bimould operator of \cite{BIM}*{Definition~4.10}.  Its coefficient map is precomposed with the quasi-shuffle derivation described in \cite{BIM}*{Proposition~4.11}.  Directly from \eqref{eq:BvI-definition},
\begin{equation*}
 E=\delta_{\mathrm{BvI}}(\one).
\end{equation*}

To define preuri, let
\[
 \mathcal D_r(F\mid z_0,\ldots,z_r)
 =\sum_{s=0}^r\frac{F(z_s)}{\prod_{k\ne s}(z_k-z_s)}.
\]
Put $C(B)_0=C(B)_1=0$ and, for $m\geq2$, put
\begin{equation}\label{eq:C-definition}
 C(B)_m=\sum_{r=1}^{m-1}\bigl(L_{m,r}(B)-R_{m,r}(B)\bigr),
\end{equation}
where
\begin{align}
 L_{m,r}(B)
 &=\mathcal D_r\left(
 t\longmapsto
 B_{m-r}\bi{X_r-t,\ldots,X_{m-1}-t}
 {Y_r,\ldots,Y_{m-1}}
 \mathrel\big|0,X_1,\ldots,X_{r-1},X_m
 \right),\label{eq:Lmr}\\
 R_{m,r}(B)
 &=\mathcal D_r\left(
 t\longmapsto
 B_{m-r}\bi{X_{r+1}-t,\ldots,X_m-t}
 {Y_{r+1},\ldots,Y_m}
 \mathrel\big|0,X_1,\ldots,X_r
 \right).\label{eq:Rmr}
\end{align}
Similar divided-difference sums, in which a bimould $C(B)$ is inserted after a position $i$ and evaluated at $X_{i+1}-X_i,\ldots,X_{i+m}-X_i$ and $Y_{i+1},\ldots,Y_{i+m}$, are studied in \cite{BKS}.  These are the arguments produced by anit. Define
\begin{equation}\label{eq:preuri-definition}
 \urit(B)=\arit(B)+\anit(C(B)),
 \qquad
 \preuri(A,B)=\urit(B)(A)+\muop(A,B).
\end{equation}
A similar product, called preuri, is studied in \cite{BKS}.  The arguments of $\preuri(A,B)$ are reversed from Burmester's convention used in \cite{BKa}*{Section~8}, i.e.
\begin{equation*}
 \preuri(A,B)=\operatorname{preuri}_{\mathrm{Bu}}(B,A).
\end{equation*}
Notice that $C(E)=0$.  For $m=2$, one has $L_{2,1}(E)=R_{2,1}(E)=-\frac12$.  When $m\geq3$, the only possible terms in \eqref{eq:C-definition} are divided differences of positive order of the constant $E_2$ or of order at least two of the affine polynomial $E_1$, and hence they vanish.  Therefore
\[
 \urit(E)=\arit(E),\qquad
 \preuri(A,E)=\preari(A,E).
\]

\begin{prop}\label{prop:preuri-comparison}
For $A,B\in\BARI^{\mathrm{pol}}$, one has
\begin{equation}\label{eq:preuri-comparison}
 \ganit_{\pic}\Bigl(
 \preari(\ganit_{\poc}(A),\ganit_{\poc}(B))
 \Bigr)
 =\preuri(A,B).
\end{equation}
In particular,
\begin{equation}\label{eq:uri-preuri}
 \uri(A,B)=\preuri(A,B)-\preuri(B,A).
\end{equation}
\end{prop}

\begin{proof}
Write
\begin{align*}
 F^{\mathrm t}
 &=F_{m-1}\bi{X_2,\ldots,X_m}{Y_2,\ldots,Y_m},&
 F^{\mathrm{t},1}
 &=F_{m-1}\bi{X_2-X_1,\ldots,X_m-X_1}{Y_2,\ldots,Y_m},\\
 F^{\mathrm h}
 &=F_{m-1}\bi{X_1,\ldots,X_{m-1}}{Y_1,\ldots,Y_{m-1}},&
 F^{\mathrm{h},m}
 &=F_{m-1}\bi{X_1-X_m,\ldots,X_{m-1}-X_m}
 {Y_1,\ldots,Y_{m-1}}.
\end{align*}
Define $\lambda(B)$ recursively by
\begin{equation}\label{eq:lambda-recursion}
 \lambda(B)_0=0,
 \qquad
 \lambda(B)_1=-B_1,
 \qquad
 \lambda(B)_m=-B_m
 +\frac{\lambda(B)^{\mathrm t}-\lambda(B)^{\mathrm{t},1}}{X_1}
 +\frac{B^{\mathrm h}-B^{\mathrm{h},m}}{X_m}.
\end{equation}

The operator $\mathcal D_r$ is symmetric in its nodes and satisfies
\begin{equation}\label{eq:divided-difference}
 \mathcal D_r(F\mid z_0,\ldots,z_r)
 =\frac{
 \mathcal D_{r-1}(F\mid z_1,\ldots,z_r)
 -\mathcal D_{r-1}(F\mid z_0,\ldots,z_{r-1})}
 {z_0-z_r}.
\end{equation}
For $r=1$, equations \eqref{eq:Lmr} and \eqref{eq:Rmr} give
\begin{equation}\label{eq:C-r-one}
 L_{m,1}(B)=\frac{B^{\mathrm h}-B^{\mathrm{h},m}}{X_m},
 \qquad
 R_{m,1}(B)=\frac{B^{\mathrm t}-B^{\mathrm{t},1}}{X_1}.
\end{equation}
When $r\geq2$, compare $L_{m-1,r-1}$ evaluated in the variables $X_2,\ldots,X_m$ and $Y_2,\ldots,Y_m$ with its value in the variables $X_2-X_1,\ldots,X_m-X_1$ and $Y_2,\ldots,Y_m$.  The two node sets are
\[
 (0,X_2,\ldots,X_{r-1},X_m)
 \quad\text{and}\quad
 (X_1,X_2,\ldots,X_{r-1},X_m).
\]
In the second set, both the nodes and the variable $t$ have been translated by $X_1$.  Order their union as $(0,X_2,\ldots,X_{r-1},X_m,X_1)$ and apply \eqref{eq:divided-difference}.  This gives
\begin{equation*}
 L_{m,r}(B)
 =\frac{L_{m-1,r-1}(B)^{\mathrm t}
       -L_{m-1,r-1}(B)^{\mathrm{t},1}}{X_1}.
\end{equation*}
The same argument for $R_{m-1,r-1}$, using the node sets $(0,X_2,\ldots,X_r)$ and $(X_1,X_2,\ldots,X_r)$, gives
\begin{equation*}
 R_{m,r}(B)
 =\frac{R_{m-1,r-1}(B)^{\mathrm t}
       -R_{m-1,r-1}(B)^{\mathrm{t},1}}{X_1}.
\end{equation*}
Here the superscripts on $L$ and $R$ mean evaluation in the two sets of variables above, as for $F^{\mathrm t}$ and $F^{\mathrm{t},1}$.  Summing these identities for $2\leq r\leq m-1$ and adding \eqref{eq:C-r-one} yields
\begin{equation}\label{eq:C-recursion}
 C(B)_1=0,
 \qquad
 C(B)_m
 =\frac{C(B)^{\mathrm t}-C(B)^{\mathrm{t},1}}{X_1}
 +\frac{B^{\mathrm{t},1}-B^{\mathrm t}}{X_1}
 +\frac{B^{\mathrm h}-B^{\mathrm{h},m}}{X_m}.
\end{equation}
Comparing \eqref{eq:lambda-recursion} and \eqref{eq:C-recursion} gives by induction
\begin{equation}\label{eq:C-lambda}
 C(B)=B+\lambda(B).
\end{equation}

Equations \eqref{eq:preuri-definition} and \eqref{eq:C-lambda} give
\begin{equation*}
 \urit(B)
 =\arit(B)+\anit(C(B))
 =\axit(B,\lambda(B)).
\end{equation*}
The recursion \eqref{eq:lambda-recursion} is the linearization at the unit of the auxiliary bimould of \cite{BKa}*{Definition~4.1}.  The resulting formula \eqref{eq:uri-preuri} was conjectured by Burmester in her thesis \cite{Bu1}*{Conjecture~5.16} and verified by her in depths at most six in \cite{Bu1}*{Proposition~5.17(i)}.  Linearizing the garit--gaxit conjugation identity in \cite{BKa}*{Proposition~5.1} gives
\begin{equation*}
 \ganit_{\pic}\circ
 \arit(\ganit_{\poc}(B))\circ\ganit_{\poc}
 =\axit(B,\lambda(B))
 =\urit(B).
\end{equation*}
Since $\ganit_{\pic}$ is multiplicative for $\muop$, \eqref{eq:preuri-comparison} follows.  Antisymmetrizing gives \eqref{eq:uri-preuri}.
\end{proof}

\cref{prop:preuri-comparison} can be seen as the polynomial form of a similar statement conjectured in \cite{BKS}.

\begin{prop}\label{prop:delta-comparison}
Put
\[
 \delta_0=\ganit_{\pic}Q\ganit_{\poc}.
\]
For every $A\in\BARI^{\mathrm{pol}}$, one has
\begin{align}
 \delta_0(A)
 &=\delta_{\mathrm{BvI}}(A)-\preuri(A,E),
 \label{eq:bare-delta}\\
 \checkdelta(A)
 &=\delta_{\mathrm{BvI}}(A)-\preuri(E,A).
 \label{eq:delta-comparison}
\end{align}
\end{prop}

\begin{proof}
The map $\delta_0$ is a derivation for $\muop$.  Its marked-block expansion gives
\begin{equation}\label{eq:delta-zero-explicit}
 (\delta_0A)_r
 =\left(\sum_{i=1}^rX_iY_i\right)A_r
  -\sum_{i=2}^rY_i(\varphi_i^-A)_r.
\end{equation}
To prove \eqref{eq:delta-zero-explicit}, observe that in a block $[a,b]$ of $\ganit_{\pic}$,
\[
 \sum_{k=a}^bX_kY_k-X_a\sum_{k=a}^bY_k
 =\sum_{k=a+1}^b(X_k-X_a)Y_k.
\]
Each factor $X_k-X_a$ cancels the corresponding denominator in $\pic$, which gives the contraction in \eqref{eq:delta-zero-explicit}.  After this cancellation, deleting the position $k$ and keeping the marked block decomposition gives exactly the terms in $(\varphi_k^-A)_r$.

The contraction formula of \cite{BIM}*{Definition~4.10} gives
\begin{equation}\label{eq:BvI-product}
 \delta_{\mathrm{BvI}}\muop(A,B)
 =\muop(\delta_{\mathrm{BvI}}A,B)
  +\muop(A,\delta_{\mathrm{BvI}}B)
  -\muop(A,E,B).
\end{equation}
We give a short proof of the last term.  Fix components $A_a$ and $B_b$.  For the single contractions the output depth is $a+b+1$.  Using local indices on the two factors, the ranges on the two sides are
\[
\begin{array}{c|c|c}
 &\delta_{\mathrm{BvI}}\muop(A,B)
 &\muop(\delta_{\mathrm{BvI}}A,B)
   +\muop(A,\delta_{\mathrm{BvI}}B)\\ \hline
 \varphi^+&A:1\leq i\leq a,\quad B:1\leq j\leq b+1
 &A:1\leq i\leq a+1,\quad B:1\leq j\leq b+1\\
 \varphi^-&A:2\leq i\leq a+1,\quad B:2\leq j\leq b+1
 &A:2\leq i\leq a+1,\quad B:2\leq j\leq b+1.
\end{array}
\]
A $B$-index $j$ is the global index $a+j$.  Every $\varphi^-$-term therefore occurs on both sides.  The $\varphi^+$-terms also agree except for one.  If $\delta_{\mathrm{BvI}}$ is applied to $A$ before concatenation, it contains the operation which deletes the last pair of $A$.  At the same position after concatenation, the operation starts in the $B$-block.  The extra term on the right is
\[
 -\frac{X_{a+1}+Y_{a+1}}2
 A_a(w_1,\ldots,w_a)B_b(w_{a+2},\ldots,w_{a+b+1}).
\]
In the local indexing of $A$, the next upper variable is the boundary value $0$.  Moving this term to the left gives the insertion of $-E_1$.  For the double contractions the output depth is $a+b+2$. In the case of $(\varphi_i^+)^2$, the range $1\leq i\leq a+b+1$ on the left is compared with the two ranges $1\leq i\leq a+1$ and $a+1\leq i\leq a+b+1$ on the right.  Their only overlap is $i=a+1$.  For $(\varphi_i^-)^2$, the ranges $2\leq i\leq a+1$ and $a+2\leq i\leq a+b+1$ partition the range on the left.  Thus the only extra double contraction on the right twice deletes the last pair of $A$, with coefficient $1/4$.  Moving it to the left gives the insertion of $-E_2$.  The multiplication term $\sum_iX_iY_i$ splits at the concatenation cut.  Summing over $a$ and $b$ proves \eqref{eq:BvI-product}.  The same calculation includes $a=0$ or $b=0$ with the empty index ranges omitted.

Since $\urit(E)$ is a derivation for $\muop$, the map $R_{E}(A)=\preuri(A,E)$ satisfies
\begin{equation*}
 R_{E}\muop(A,B)
 =\muop(R_{E}A,B)+\muop(A,R_{E}B)-\muop(A,E,B).
\end{equation*}
It follows that $L=\delta_{\mathrm{BvI}}-\preuri(-,E)$ is a derivation for $\muop$.

For a polynomial $p(X,Y)$, let $P(p)$ be concentrated in depth one with $P(p)_1=p$, and put $S_j=Y_1+\cdots+Y_j$.  In depth two, direct substitution gives
\begin{align*}
 (\delta_{\mathrm{BvI}}P(p))_2
 ={}&-\frac12(X_1-X_2+Y_1)p(X_2,S_2)
     -\frac12(X_2+Y_2)p(X_1,Y_1)\\
 &-\frac12(X_1-X_2+Y_2)p(X_1,S_2),\\
 (\preuri(P(p),E))_2
 ={}&-\frac12(X_1-X_2+Y_1)p(X_2,S_2)
     -\frac12(X_2+Y_2)p(X_1,Y_1)\\
 &-\frac12(X_1-X_2-Y_2)p(X_1,S_2).
\end{align*}
The only difference is $-Y_2p(X_1,S_2)$.  In depth three, both expressions are equal to
\begin{equation*}
 \frac14\bigl(
 p(X_3,S_3)+p(X_1,Y_1)-p(X_1,S_3)
 \bigr).
\end{equation*}
Both vanish in depths at least four.  For $\delta_{\mathrm{BvI}}$ this follows from its contraction formula.  In the preuri case, it follows from $C(E)=0$ and the depth supports of $P(p)$ and $E$.  Together with the depth-one component, this gives
\begin{equation}\label{eq:generator-comparison}
 L(P(p))=\delta_0(P(p)),
\end{equation}
where both sides have depth-one component $X_1Y_1p(X_1,Y_1)$, depth-two component $-Y_2p(X_1,Y_1+Y_2)$ and vanish in all other depths. For a fixed output depth $r$, truncate $A$ above depth $r$.  Each component of this truncation is a finite sum of products $\muop(P(p_1),\ldots,P(p_d))$.  Both sides of \eqref{eq:bare-delta} at depth $r$ depend only on this truncation.  Their derivation property proves \eqref{eq:bare-delta}. Finally, \cref{prop:preuri-comparison} gives
\begin{align*}
 \checkdelta(A)
 &=\delta_0(A)+\uri(A,E)\\
 &=\delta_{\mathrm{BvI}}(A)-\preuri(A,E)
   +\preuri(A,E)-\preuri(E,A),
\end{align*}
which is \eqref{eq:delta-comparison}.
\end{proof}

\section{Swap invariance}

Let
\begin{equation*}
 \mathfrak E(u,v)=\frac1u+\frac1v,
 \qquad
 \mathcal P=\invmu(\one-\mathfrak E),
 \qquad
 \mathcal P_r=\prod_{i=1}^r
 \left(\frac1{X_i}+\frac1{Y_i}\right).
\end{equation*}
The flexion unit $\mathfrak E$ is denoted by $E$ in \cite{BKa}*{Section~7}.  It should not be confused with the bimould $E$ in \eqref{eq:E-definition}. The bimould $\mathfrak E$ is concentrated in depth one and is odd, symmetric and satisfies the tripartite identity
\[
 \mathfrak E(u_1,v_1)\mathfrak E(u_2,v_2)
 =\mathfrak E(u_1+u_2,v_1)\mathfrak E(u_2,v_2-v_1)
 +\mathfrak E(u_1+u_2,v_2)\mathfrak E(u_1,v_1-v_2).
\]
Therefore, it is a flexion unit equal to its conjugate unit.  In this section we allow more poles than in $\cL_r$ and work in the ring
\begin{equation}\label{eq:more-poles}
 \widehat{\cL}_r(\QQ):=\QQ[[X_1,Y_1,\ldots,X_r,Y_r]]
 \left[\frac1\ell\ \middle|\
 0\ne\ell\in\QQ X_1+\cdots+\QQ X_r+
                    \QQ Y_1+\cdots+\QQ Y_r\right],
\end{equation}
in which every nonzero rational linear form in the variables is inverted. These are the rings used in \cite{BKa}*{Section~7}.  They form a family of functions in the sense of Furusho, Hirose and Komiyama \cite{FHK}*{Definition~1}, as explained in \cite{BKa}*{Remark~7.3}.  In particular, the flexion operations and the depthwise constructions used below are defined over $\widehat{\cL}(\QQ)$.

Since $\mathfrak E$ equals its conjugate unit, the primary bimould of the conjugate unit is $\mathcal P$.  A bimould $H$ is therefore $\mathfrak E$-alternal if $\ganit_{\mathcal P}^{-1}(H)$ is alternal \cite{Kaw2}*{Definition~2.2}.  In the present symmetric case, Kawamura's dimorphic space is the following, where $\ARI$ is taken with components in \eqref{eq:more-poles}:
\begin{equation*}
 \ARI_{\underline{\mathrm{al}}/\underline{\mathfrak{ol}}}
 =\left\{
 C\in\ARI\ \middle|\
 C\text{ is alternal},\quad
 \swapop(C)\text{ is $\mathfrak E$-alternal},\quad
 C_1\bi{-X_1}{-Y_1}=C_1\bi{X_1}{Y_1}
 \right\}.
\end{equation*}
This is the symmetric specialization of the dimorphic spaces in \cite{Kaw2}*{Section~3}.  By \cite{BKa}*{Remark~7.3}, the proofs of \cite{Kaw2}*{Theorem~4.2, Proposition~5.1, Lemma~5.2 and Corollary~5.7} remain valid over $\widehat{\cL}(\QQ)$.

\begin{lem}\label{lem:fixed-locus}
Let $A$ be alternil and put $C=\ganit_{\poc}(A)$.  Then
\begin{equation}\label{eq:fixed-locus}
 \swapop(A)=A
 \quad\Longleftrightarrow\quad
 \swapop(C)=\ganit_{\mathcal P}(C).
\end{equation}
If these conditions hold and $A_1$ is even, then $C$ belongs to Kawamura's dimorphic space $\ARI_{\underline{\mathrm{al}}/\underline{\mathfrak{ol}}}$ for the unit $\mathfrak E$.
\end{lem}

\begin{proof}
The marked-word identities used below are those of \cite{BKa}*{Lemma~7.2 and Equations~(7.3)--(7.4)}. Applying the marked-word formula for ganit to $\poc$, $\picy$ and $\mathcal P$ gives
\[
 \swapop\circ\ganit_{\poc}
 =\ganit_{\picy}\circ\swapop,
 \qquad
 \ganit_{\mathcal P}=\ganit_{\picy}\circ\ganit_{\pic}.
\]
Since $C=\ganit_{\poc}(A)$, these identities give
\begin{equation*}
 \swapop(C)=\ganit_{\picy}(\swapop(A)),
 \qquad
 \ganit_{\mathcal P}(C)=\ganit_{\picy}(A).
\end{equation*}
The equivalence follows because $\ganit_{\picy}$ is invertible. Also, $C$ is alternal by the alternal--alternil transport \cite{Bu1}*{Proposition~5.9}, and \eqref{eq:fixed-locus} gives
\[
 \ganit_{\mathcal P}^{-1}(\swapop(C))=C.
\]
It follows that $\swapop(C)$ is $\mathfrak E$-alternal.  Finally, $\ganit_{\poc}$ is the identity in depth one, so the depth-one parity is preserved.  These are the defining conditions of the dimorphic space above.
\end{proof}

\begin{lem}\label{lem:E-push}
Let $C\in\ARI_{\underline{\mathrm{al}}/ \underline{\mathfrak{ol}}}$, and put
\[
 S=\swapop(C),
 \qquad
 U=\pushop(S).
\]
Then
\begin{equation}\label{eq:one-bimould}
 \muop(U,\mathcal P)+\anit(U)(\mathcal P)
 =\muop(\mathcal P,S)+\amit(S)(\mathcal P).
\end{equation}
\end{lem}

\begin{proof}
Put $R=\one-\mathfrak E$, so that $\mathcal P=\invmu(R)$. Write
\[
 \operatorname{swamu}(A,B)
 =\swapop\bigl(\muop(\swapop(A),\swapop(B))\bigr).
\]
By \cite{Kaw2}*{Theorem~4.2}, the bimould $C$ is $\mathfrak E$-push invariant.  Applying the inverse formula in \cite{Kaw2}*{Corollary~5.7}, then applying swap and multiplying on the right by $R$, gives
\begin{equation}\label{eq:E-push-raw}
 \muop(S,R)=\muop(R,U)
 +\pushop\bigl(\muop(\mathfrak E,S)\bigr)
 -\pushop\bigl(\operatorname{swamu}(S,\mathfrak E)\bigr).
\end{equation}
Since $\mathfrak E$ is concentrated in depth one and $\pushop(\mathfrak E)=-\mathfrak E$, the flexion definitions give
\[
 \amit(S)(\mathfrak E)
 =-\pushop\bigl(\muop(\mathfrak E,S)\bigr).
\]
Proposition~5.1 and Lemma~5.2 of \cite{Kaw2} similarly give
\[
 \anit(U)(\mathfrak E)
 =-\pushop\bigl(\operatorname{swamu}(S,\mathfrak E)\bigr).
\]
\cref{eq:E-push-raw} now becomes
\begin{equation}\label{eq:E-push}
 \muop(S,R)=\muop(R,U)
 -\amit(S)(\mathfrak E)+\anit(U)(\mathfrak E).
\end{equation}
Let $L=\axit(S,-U)=\amit(S)-\anit(U)$.  Equation \eqref{eq:E-push} gives
\[
 L(\mathfrak E)=\muop(R,U)-\muop(S,R).
\]
Since $L$ is a derivation for $\muop$,
\begin{align*}
 L(\mathcal P)
 &=\muop\bigl(\mathcal P,L(\mathfrak E),\mathcal P\bigr)\\
 &=\muop(U,\mathcal P)-\muop(\mathcal P,S).
\end{align*}
Writing out $L=\amit(S)-\anit(U)$ gives \eqref{eq:one-bimould}.
\end{proof}

\begin{cor}\label{cor:strong-conjugation}
In addition to the assumptions of \cref{lem:E-push}, suppose that
\begin{equation*}
 S=\ganit_{\mathcal P}(C).
\end{equation*}
Then
\begin{equation*}
 \ganit_{\mathcal P}\circ\arit(C)\circ\ganit_{\mathcal P}^{-1}
 =\axit(S,-U).
\end{equation*}
\end{cor}

\begin{proof}
The GAXI composition law \cite{Kaw1}*{Proposition~3.10} writes the left-hand side as $\axit(S,C')$, where $C'$ is uniquely determined by
\begin{equation}\label{eq:C-prime}
 \muop(C',\mathcal P)+\axit(S,C')(\mathcal P)
 =-\muop(\mathcal P,S).
\end{equation}
Since $\mathcal P_0=1$, the depth-$n$ part of $\muop(C',\mathcal P)$ contains $C'_n$, while the other terms containing $C'$ use only lower-depth components.  Thus \eqref{eq:C-prime} determines $C'$ recursively.  Equation \eqref{eq:one-bimould} shows that $C'=-U$ satisfies it.
\end{proof}

\begin{lem}\label{lem:E}
The bimould $E$ is alternil and swap invariant.
\end{lem}

\begin{proof}
Swap invariance follows directly from \eqref{eq:E-definition}.  For alternility, its only nonzero coefficients are
\[
 a_E(y_{2,0})=a_E(y_{1,1})=-\frac12,
 \qquad
 a_E(y_{1,0}y_{1,0})=\frac14.
\]
The coefficient map of $E$ is supported in weight two.  Since the bi-stuffle product is weight homogeneous and $y_{1,0}$ is the only nonempty word of weight one, the only product of two nonempty words on which this coefficient map can be nonzero is
\[
 y_{1,0}*y_{1,0}
 =2y_{1,0}y_{1,0}+y_{2,0},
\]
and its value is $2\cdot\frac14-\frac12=0$.
\end{proof}

\begin{rem}\label{rem:E}
The bimould $E$ is the unique alternil swap invariant polynomial bimould of weight two up to scalars.  Namely, such a bimould has the form $(aX_1+bY_1,c,0,\ldots)$, swap invariance gives $a=b$, and the computation in the proof of \cref{lem:E} gives $2c+a=0$, so it is a multiple of $E$.  It is excluded from $\mathfrak B$ only by the depth-one parity.  The correction by $E$ in \eqref{eq:delta-definition} plays the role of the weight-two correction term in the Serre derivative of modular forms.
\end{rem}

\begin{prop}\label{prop:preuri-swap}
For every $A\in\mathfrak B$, one has
\begin{equation*}
 \swapop\bigl(\preuri(E,A)\bigr)=\preuri(E,A).
\end{equation*}
\end{prop}

\begin{proof}
Put
\[
 C=\ganit_{\poc}(A),
 \qquad
 C_0=\ganit_{\poc}(E),
 \qquad
 S=\swapop(C),
 \qquad
 U=\pushop(S).
\]
Applying \cref{lem:fixed-locus} first to $A$ and then to $E$, using \cref{lem:E} in the second case, gives
\begin{equation}\label{eq:two-fixed-points}
 S=\ganit_{\mathcal P}(C),
 \qquad
 \swapop(C_0)=\ganit_{\mathcal P}(C_0).
\end{equation}

The swap formula for preari \cite{Sc}*{Equation~(2.4.10)} gives
\begin{equation}\label{eq:swap-preari}
 \swapop\bigl(\preari(C_0,C)\bigr)
 =\muop(\swapop(C_0),S)
  +\axit(S,-U)(\swapop(C_0)).
\end{equation}
By \cref{cor:strong-conjugation}, the multiplicativity of $\ganit_{\mathcal P}$ \cite{Kaw1}*{Proposition~3.12}, and \eqref{eq:two-fixed-points}, the right-hand side is
\[
 \ganit_{\mathcal P}\bigl(
  \muop(C_0,C)+\arit(C)(C_0)
 \bigr)
 =\ganit_{\mathcal P}\bigl(\preari(C_0,C)\bigr).
\]
The fixed-locus identities also give, as stated after \cite{BKa}*{Lemma~7.2},
\begin{equation}\label{eq:fixed-map}
 \ganit_{\poc}\circ\swapop\circ\ganit_{\pic}
 =\ganit_{\mathcal P}^{-1}\circ\swapop.
\end{equation}
It follows from \eqref{eq:swap-preari} and \eqref{eq:fixed-map} that $\ganit_{\pic}(\preari(C_0,C))$ is swap invariant. The bimould is identified with $\preuri(E,A)$ by \cref{prop:preuri-comparison}.
\end{proof}

\begin{proof}[Proof of \cref{thm:main}]
By \cite{BKa}*{Theorem~A(ii)}, $(\mathfrak B,\uri)$ is a weight-graded Lie algebra.  It remains, by \cref{prop:ambient-triple}, to show that $\checkdelta$ preserves $\mathfrak B$ and raises the weight by two.  Formula \eqref{eq:delta-comparison} and the defining formulas \eqref{eq:BvI-definition} and \eqref{eq:preuri-definition} show that $\checkdelta(A)$ is polynomial.  Suppose that $A$ has maximum depth $n$ and maximum ordinary degree $d$.  A term of $C(A)_m$ can be nonzero only if its divided-difference order $r$ in \eqref{eq:C-definition} satisfies $r\geq m-n$ and $r\leq d$. We obtain $C(A)_m=0$ for $m>n+d$, so finite depth is also preserved.

If $A$ is alternil, then $\ganit_{\poc}(A)$ is alternal. Multiplication by $\sum_iX_iY_i$ preserves alternality, since this sum is unchanged by a shuffle of the variables.  This proves that $\delta_0(A)$ is alternil. By \cref{lem:E}, the bimould $E$ is alternil, and the alternil bimoulds form a uri Lie algebra by transport from ari.  Hence $\uri(A,E)$ is alternil.  In depth one, the uri bracket vanishes and
\begin{equation*}
 (\checkdelta A)_1=X_1Y_1A_1.
\end{equation*}
The required depth-one parity is preserved.

Formula \eqref{eq:delta-comparison} gives
\[
 \checkdelta(A)=\delta_{\mathrm{BvI}}(A)-\preuri(E,A).
\]
The operator $\delta_{\mathrm{BvI}}$ commutes with swap by \cite{BIM}*{Proposition~4.11}, and the second term is swap invariant by \cref{prop:preuri-swap}.  Hence $\checkdelta(A)$ is swap invariant.  Finally, both terms in \eqref{eq:delta-definition} have weight two.  This completes the proof.
\end{proof}

\section{The kernel of \texorpdfstring{$\checkD$}{D-check} and the operations
\texorpdfstring{$\uri_n$}{uri-n}}\label{sec:kernel}

In this section we prove \cref{thm:B,thm:C} and study the operations $\uri_n$ on $\mathfrak m$.  Everything follows from the relations \eqref{eq:sl2-relations} together with \cref{thm:main} and the fact that the weight grading of $\mathfrak B$ is by positive integers.

\begin{lem}\label{lem:sl2-identities}
Let $A\in\mathfrak B_k$.
\begin{enumerate}
\item[(i)] If $\checkD(A)=0$, then
$\checkD\checkdelta^m(A)=m(k+m-1)\,\checkdelta^{m-1}(A)$ for all $m\geq1$.
\item[(ii)] If $\checkdelta(A)=0$, then
$\checkdelta\checkD^j(A)=-j(k-j+1)\,\checkD^{j-1}(A)$ for all $j\geq1$.
\end{enumerate}
\end{lem}

\begin{proof}
The relations \eqref{eq:sl2-relations} give
\[
 [\checkD,\checkdelta^m]
 =\sum_{i=0}^{m-1}
 \checkdelta^i\,[\checkD,\checkdelta]\,\checkdelta^{m-1-i}
 =\sum_{i=0}^{m-1}\checkdelta^i\,\checkW\,\checkdelta^{m-1-i}\,.
\]
Applied to $A$, the operator $\checkW$ in the $i$-th summand acts on an element of weight $k+2(m-1-i)$.  Summing these scalars over $i$ gives $mk+m(m-1)$, so $[\checkD,\checkdelta^m](A)=m(k+m-1)\,\checkdelta^{m-1}(A)$, which proves (i).  The same computation with $[\checkdelta,\checkD^j] =-\sum_{i=0}^{j-1}\checkD^i\,\checkW\,\checkD^{j-1-i}$ proves (ii).
\end{proof}

\begin{proof}[Proof of \cref{thm:B}]
Since $\checkD$ is a derivation of uri, for $A,B\in\mathfrak m$ we have
\[
 \checkD\uri(A,B)
 =\uri(\checkD A,B)+\uri(A,\checkD B)=0.
\]
This proves that $\mathfrak m$ is closed under uri.  Its weight grading follows because $\checkD$ has weight $-2$, which proves (i).

We next prove the injectivity in (ii).  Let $A\in\mathfrak B_k$ be homogeneous with $\checkdelta(A)=0$ and $A\neq0$, so $k\geq1$.  Since $\checkD$ lowers the weight by two and $\mathfrak B_w=0$ for $w\leq0$, there is a minimal $j\geq1$ with $\checkD^j(A)=0$.  Part (ii) of \cref{lem:sl2-identities} gives
\[
 0=\checkdelta\checkD^j(A)=-j(k-j+1)\,\checkD^{j-1}(A)\,,
\]
and $\checkD^{j-1}(A)\neq0$, so $k=j-1$.  The element $\checkD^{j-1}(A)$ has weight $k-2(j-1)=-k<0$, so it vanishes, which contradicts the minimality of $j$.

Next, the sum $\sum_{m\geq0}\checkdelta^m(\mathfrak m)$ is direct.  Suppose that $\sum_{m=0}^{M}\checkdelta^m(x_m)=0$ with homogeneous $x_m\in\mathfrak m_{k-2m}$ and $x_M\neq0$.  Applying part (i) of \cref{lem:sl2-identities} repeatedly gives $\checkD^M\checkdelta^m(x_m)=0$ for $m<M$ and
\[
 \checkD^M\checkdelta^M(x_M)
 =M!\,(k'+M-1)(k'+M-2)\cdots k'\,x_M\,,
\]
where $k'=k-2M$ is the weight of $x_M$.  All factors are positive, since $k'\geq1$.  Therefore $x_M=0$, a contradiction.

Now we show $\mathfrak B_w=\checkdelta(\mathfrak B_{w-2})+\mathfrak m_w$ by induction on $w$.  For $w\leq2$ we have $\checkD(\mathfrak B_w)\subseteq\mathfrak B_{w-2}=0$, so $\mathfrak B_w=\mathfrak m_w$.  Let $A\in\mathfrak B_w$.  Iterating the induction hypothesis in weights below $w$, we can write $\checkD(A)=\sum_{m\geq0}\checkdelta^m(y_m)$ with homogeneous $y_m\in\mathfrak m_{k_m}$ of weight $k_m=w-2-2m$, where $y_m=0$ unless $k_m\geq1$.  Put
\[
 A'=\sum_{m\geq0}\frac{1}{(m+1)(k_m+m)}\,
 \checkdelta^{m+1}(y_m)\,.
\]
The denominators are nonzero, since $k_m\geq1$ whenever $y_m\neq0$, and part (i) of \cref{lem:sl2-identities} gives $\checkD(A')=\checkD(A)$.  So $A-A'\in\mathfrak m_w$, which together with the directness proves the decomposition in (iii).  The same argument proves the surjectivity in (ii), since every element of $\mathfrak B_{w-2}$ is of the form $\checkD(A')$ as above.  Injectivity in (ii) was proved above.  Finally, the dimension formula follows from $\dim_\QQ\mathfrak m_w =\dim_\QQ\mathfrak B_w-\operatorname{rank}(\checkD|_{\mathfrak B_w})$.
\end{proof}

The extra term coming from $G_2$ in the decomposition of $\widetilde{\mathcal M}$ has no counterpart in \cref{thm:B}.  In $\widetilde{\mathcal M}$ the constants sit in weight zero, the scalar $(m+1)(k_m+m)$ in the proof of \cref{thm:B} vanishes for $k_m=m=0$, and the derivatives of $G_2$ appear as a separate summand.  The Lie algebra $\mathfrak B$ has no weight-zero part, so no extra term appears.

Notice that $\uri_0(f,g)=\uri(f,g)$.

\begin{proof}[Proof of \cref{thm:C}]
Since $\checkD$ is a derivation of uri, part (i) of \cref{lem:sl2-identities} gives
\[
 \checkD\,\uri\bigl(\checkdelta^r(f),\checkdelta^s(g)\bigr)
 ={}r(k+r-1)\,
 \uri\bigl(\checkdelta^{r-1}(f),\checkdelta^s(g)\bigr)
 +s(l+s-1)\,
 \uri\bigl(\checkdelta^r(f),\checkdelta^{s-1}(g)\bigr)\,.
\]
Here and below, a summand containing a negative exponent is omitted. In $\checkD\,\uri_n(f,g)$, the term $\uri(\checkdelta^r(f),\checkdelta^s(g))$ with $r+s=n-1$ therefore appears with the coefficient
\[
 (-1)^{r+1}\binom{n+k-1}{s}\binom{n+l-1}{r+1}(r+1)(k+r)
 +(-1)^{r}\binom{n+k-1}{s+1}\binom{n+l-1}{r}(s+1)(l+s)\,.
\]
We have $\binom{n+l-1}{r+1}(r+1)=\binom{n+l-1}{r}(n+l-1-r)$ with $n+l-1-r=l+s$, and $\binom{n+k-1}{s+1}(s+1)=\binom{n+k-1}{s}(n+k-1-s)$ with $n+k-1-s=k+r$.  So the two summands cancel, and $\checkD\,\uri_n(f,g)=0$. For the symmetry, interchange $r$ and $s$ in \eqref{eq:uri-n} and use $(-1)^s=(-1)^{n}(-1)^{r}$ together with the antisymmetry of uri.
\end{proof}

Notice that the parity in \cref{thm:C} is opposite to the classical case.  Since uri is antisymmetric, the operations $\uri_n$ with odd $n$ are symmetric.  In particular, $\uri_1(f,f)=2k\,\uri(f,\checkdelta(f))$ does not vanish in general. For example, with $\xi_1=(1,0,0,\ldots)\in\mathfrak B_1$, the element $\frac12\uri_1(\xi_1,\xi_1) =\uri(\xi_1,\checkdelta(\xi_1))$ has the depth components
\[
 X_1Y_1-2Y_1X_2+2X_1Y_2-X_2Y_2\,,
 \qquad
 -X_1+2X_2+Y_2-X_3-Y_3
\]
in depths two and three.

\begin{lem}\label{lem:uri-n-generation}
For $f\in\mathfrak m_k$, $g\in\mathfrak m_l$ and $N\geq0$, one has
\[
 \operatorname{span}_{\QQ}\{\uri(\checkdelta^a(f),\checkdelta^{N-a}(g))\mid0\leq a\leq N\}
 =\operatorname{span}_{\QQ}\{\checkdelta^{N-n}\uri_n(f,g)\mid0\leq n\leq N\}.
\]
\end{lem}

\begin{proof}
The claim is immediate if $f=0$, $g=0$ or $N=0$, so suppose that both elements are nonzero and $N\geq1$.  Consider
\[
 z_n(f,g)=\sum_{a+b=n}(-1)^a
 \binom{n+k-1}{b}\binom{n+l-1}{a}
 \checkdelta^a(f)\otimes\checkdelta^b(g)\,.
\]
On the tensor product, let $\checkD$ and $\checkdelta$ act by the sum of their actions on the two factors.  The calculation in the proof of \cref{thm:C} gives $\checkD(z_n(f,g))=0$.  Put
\[
V_N=\operatorname{span}\{\checkdelta^a(f)\otimes
 \checkdelta^{N-a}(g)\mid 0\leq a\leq N\}.
\]
By the injectivity of $\checkdelta$ in \cref{thm:B}, the displayed vectors are nonzero and lie in distinct pairs of tensor-factor weights, so they form a basis of $V_N$.  With the analogous basis of $V_{N-1}$, the restriction of $\checkD$ to the span of the vectors with $0\leq a<N$ is triangular with nonzero diagonal entries $(N-a)(l+N-a-1)$ for $0\leq a<N$.  The map $\checkD:V_N\to V_{N-1}$ is therefore surjective, and its kernel is spanned by $z_N(f,g)$, which is nonzero since its coefficient at $a=0$ is $\binom{N+k-1}{N}$.  We also have, for $n<N$,
\[
 \checkD\checkdelta^{N-n}z_n(f,g)
 =(N-n)(k+l+N+n-1)\checkdelta^{N-n-1}z_n(f,g).
\]
Induction on $N$ now shows that the vectors $\checkdelta^{N-n}z_n(f,g)$ for $0\leq n\leq N$ form another basis of $V_N$.

The bracket map $u\otimes v\mapsto\uri(u,v)$ intertwines the diagonal $\mathfrak{sl}_2$-action with the action on $\mathfrak B$, since $\checkD$, $\checkdelta$ and $\checkW$ are derivations.  It sends $z_n(f,g)$ to $\uri_n(f,g)$.  Applying this map to the two bases of $V_N$ gives the claimed equality of spans.
\end{proof}

Let $\operatorname{pr}_{\mathfrak m}$ be the projection onto $\mathfrak m$ given by \cref{thm:B}(iii).

\begin{prop}\label{prop:uri-n-projection}
For $f\in\mathfrak m_k$, $g\in\mathfrak m_l$ and $r+s=n$, one has
\begin{equation}\label{eq:uri-n-projection}
 \operatorname{pr}_{\mathfrak m}
 \uri\bigl(\checkdelta^r(f),\checkdelta^s(g)\bigr)
 =
 \frac{(-1)^r}{\binom{2n+k+l-2}{n}}\uri_n(f,g).
\end{equation}
\end{prop}

\begin{proof}
The formula is immediate if $f=0$, $g=0$ or $n=0$, so suppose that both elements are nonzero and $n\geq1$.
Define a linear functional on $V_n$ by
\[
 \lambda\bigl(\checkdelta^a(f)\otimes
 \checkdelta^{n-a}(g)\bigr)=(-1)^a.
\]
This functional vanishes on $\checkdelta(V_{n-1})$, while the Vandermonde identity gives
\[
 \lambda\bigl(z_n(f,g)\bigr)
 =\sum_{a+b=n}
 \binom{n+k-1}{b}\binom{n+l-1}{a}
 =\binom{2n+k+l-2}{n}.
\]
The coefficient of $z_n(f,g)$ in $\checkdelta^r(f)\otimes\checkdelta^s(g)$ is the scalar in \eqref{eq:uri-n-projection}.  Applying $u\otimes v\mapsto\uri(u,v)$ proves the formula.
\end{proof}

\section{The odd derivations}\label{sec:odd-derivations}

Given $r\geq0$, define the inner derivation
\begin{equation}\label{eq:partial-odd}
 \partial_{2r+1}
 =\operatorname{ad}_{\uri}\bigl(\checkdelta^r(\xi_1)\bigr),
 \qquad
 \partial_{2r+1}(A)
 =\uri\bigl(\checkdelta^r(\xi_1),A\bigr).
\end{equation}
We set
\[
 \mathfrak d=\ker(\partial_1).
\]
Notice that $\partial_1$ is not one of the three operators of \cref{thm:main}.  It has weight one, commutes with $\checkD$ and has the following explicit form, in the notation of the proof of \cref{prop:preuri-comparison} with $m=r$, where $\lambda$ is defined by \eqref{eq:lambda-recursion}.

\begin{lem}\label{lem:partial-one}
For $A\in\BARI^{\mathrm{pol}}$ and $r\geq2$, one has
\[
 (\partial_1A)_r
 =A^{h,r}-A^h+A^t+\lambda(A)^{t,1}
 -\sum_{i=1}^{r-1}\bigl(\varphi_i^+A-\varphi_{i+1}^-A\bigr)_r,
\]
and $(\partial_1A)_1=0$.
\end{lem}

\begin{proof}
Since $\xi_1$ is constant in depth one, all divided differences in \eqref{eq:C-definition} vanish, so $C(\xi_1)=0$ and $\urit(\xi_1)=\arit(\xi_1)$.  In the linearization of gaxit, a factor $(\xi_1)_1=1$ occupies a block of length two.  For amit the marked letter is the second one, which deletes $X_i$ and merges $Y_i+Y_{i+1}$, and for anit it is the first one, which deletes $X_{i+1}$ and merges $Y_i+Y_{i+1}$.  Hence $\arit(\xi_1)(A)_r=\sum_{i=1}^{r-1}(\varphi_i^+A-\varphi_{i+1}^-A)_r$. On the other hand, $\xi_1$ has depth one, so in $\urit(A)(\xi_1)$ the whole word is one block.  Its marked letter is the last one for amit and the first one for anit, which gives
\[
 \urit(A)(\xi_1)_r
 =A^{h,r}-A^{t,1}+C(A)^{t,1}
 =A^{h,r}+\lambda(A)^{t,1}
\]
by \eqref{eq:C-lambda}.  Finally, we have $\muop(\xi_1,A)_r=A^t$ and $\muop(A,\xi_1)_r=A^h$.  Inserting these into
\[
 \uri(\xi_1,A)=\preuri(\xi_1,A)-\preuri(A,\xi_1)
\]
proves the formula.
\end{proof}

\begin{proof}[Proof of \cref{thm:D}]
Inner derivations are derivations.  The map $\partial_{2r+1}$ is homogeneous of weight $2r+1$.  In particular, $\mathfrak d$ is a weight-graded Lie subalgebra. For every derivation $T$ of a Lie algebra and every element $x$, one has
\[
 [T,\operatorname{ad}_{\uri}(x)]
 =\operatorname{ad}_{\uri}(T(x)).
\]
The element $\xi_1$ belongs to $\mathfrak m_1$, so \cref{lem:sl2-identities} gives
\[
 \checkD\checkdelta^r(\xi_1)=r^2\checkdelta^{r-1}(\xi_1).
\]
Weight homogeneity and the definition of $\partial_{2r+1}$ give the other two commutator formulas.

The depth-two component of $\partial_1(\checkdelta(\xi_1))$ given above shows that $\partial_1$ is nonzero. Commuting $\partial_{2r+1}$ with $\checkD$ a total of $r$ times gives $(r!)^2\partial_1$.  Hence all the $\partial_{2r+1}$ are nonzero, and their distinct weights show that they are linearly independent.  The commutator formulas identify their span with the lowest-weight Verma module of lowest weight one.
\end{proof}

Equivalently, the whole family is given by the formal identity
\[
 \sum_{r\geq0}\frac{t^r}{r!}\,\partial_{2r+1}
 =e^{t\checkdelta}\partial_1e^{-t\checkdelta}
 =\operatorname{ad}_{\uri}\bigl(e^{t\checkdelta}\xi_1\bigr).
\]

Combining \eqref{eq:partial-odd} with \cref{prop:uri-n-projection} gives
\begin{equation}\label{eq:partial-uri-n}
 \operatorname{pr}_{\mathfrak m}\bigl(\partial_{2r+1}(A)\bigr)
 =\frac{(-1)^r}{\binom{2r+k-1}{r}}\,
  \uri_r(\xi_1,A)
 \qquad(A\in\mathfrak m_k).
\end{equation}
So, up to the scalar, $\uri_r(\xi_1,-)$ is the component in $\mathfrak m$ of the restriction of $\partial_{2r+1}$ to $\mathfrak m$.

\section{The double shuffle Lie algebra}\label{sec:dm0}

We first recall Racinet's double shuffle Lie algebra $\mathfrak{dm}_0$ (see \cite{Rac} and \cite{Bu1}*{Appendix~B.2}).  Let $\mathsf X=\{x_0,x_1\}$ and $\mathsf Y=\{y_1,y_2,\ldots\}$, where $y_n=x_0^{n-1}x_1$.  The coproducts $\Delta_{\!\shuffle}$ on $\QQ\langle\mathsf X\rangle$ and $\Delta_*$ on $\QQ\langle\mathsf Y\rangle$ are the algebra morphisms given by
\[
 \Delta_{\!\shuffle}(x_i)=x_i\otimes1+1\otimes x_i,
 \qquad
 \Delta_*(y_n)=y_n\otimes1+1\otimes y_n
  +\sum_{j=1}^{n-1}y_j\otimes y_{n-j}.
\]
Let
\[
 \Pi_Y:\QQ\langle\mathsf X\rangle
 \longrightarrow\QQ\langle\mathsf Y\rangle
\]
be the linear map with $\Pi_Y(1)=1$ which kills the words ending in $x_0$ and sends
\[
 x_0^{s_1-1}x_1\cdots x_0^{s_r-1}x_1
 \longmapsto y_{s_1}\cdots y_{s_r}.
\]
For $\psi\in\QQ\langle\mathsf X\rangle$, put
\[
 \psi_*=\Pi_Y(\psi)
 +\sum_{n\geq2}\frac{(-1)^{n-1}}{n}
  \bigl(\Pi_Y(\psi)\mid y_n\bigr)y_1^n,
\]
where $(\psi\mid w)$ denotes the coefficient of the word $w$ in $\psi$.  Then $\mathfrak{dm}_0$ consists of the polynomials $\psi$ which satisfy
\[
 \begin{gathered}
  (\psi\mid x_0)=(\psi\mid x_1)=(\psi\mid x_0x_1)=0,\\
  \Delta_{\!\shuffle}(\psi)=\psi\otimes1+1\otimes\psi,
  \qquad
  \Delta_*(\psi_*)=\psi_*\otimes1+1\otimes\psi_*.
 \end{gathered}
\]
Its weight is the total degree in $x_0,x_1$.  By Racinet, $\mathfrak{dm}_0$ is a weight-graded Lie algebra for the Ihara bracket
\[
 \{\psi,\phi\}_{\mathrm{Ih}}
 =d_\psi(\phi)-d_\phi(\psi)+[\psi,\phi],
 \qquad
 d_\psi(x_0)=0,\quad d_\psi(x_1)=[x_1,\psi]
\]
(see \cite{Bu1}*{Theorem~B.30}).

We now construct a bimould directly from an element of $\mathfrak{dm}_0$.  Let $\psi\in(\mathfrak{dm}_0)_k$ be homogeneous, and write
\[
 c_\psi(s_1,\ldots,s_r)
 =\bigl(\Pi_Y(\psi)\mid y_{s_1}\cdots y_{s_r}\bigr),
 \qquad a_\psi=c_\psi(k).
\]
Set $\iota(\psi)_0=0$ and, for $r\geq1$, define
\begin{align}\label{eq:dm-to-m}
 &\iota(\psi)_r
 \bi{X_1,\ldots,X_r}{Y_1,\ldots,Y_r}\notag\\
 &\quad=
 (-1)^{k-1}
 \sum_{\substack{s_1+\cdots+s_r=k\\s_i\geq1}}
 c_\psi(s_1,\ldots,s_r)
 \left(
  \prod_{i=1}^rX_i^{s_i-1}
  +\prod_{i=1}^r
   (Y_1+\cdots+Y_{r+1-i})^{s_i-1}
 \right)
 +\begin{cases}
   \dfrac{1}{k}a_\psi,&r=k,\\[4pt]
   0,&r\neq k.
  \end{cases}
\end{align}
Extend this definition linearly to $\mathfrak{dm}_0$.

Adjoin a formal central generator $\sigma_1$ of weight one to $\mathfrak{dm}_0$:
\begin{equation}\label{eq:dm-plus}
 \mathfrak{dm}_0^+
 =\QQ\sigma_1\oplus\mathfrak{dm}_0,
 \qquad
 \{\sigma_1,\psi\}_{\mathrm{Ih}}=0
 \quad(\psi\in\mathfrak{dm}_0^+).
\end{equation}
We extend $\iota$ by $\iota(\sigma_1)=\xi_1$.

On polynomial bimoulds, let the two Euler operators act componentwise by
\[
 (\EX A)_r=\left(\sum_{i=1}^rX_i\partial_{X_i}\right)A_r,
 \qquad
 (\EY A)_r=\left(\sum_{i=1}^rY_i\partial_{Y_i}\right)A_r.
\]
A monomial of total $X$-degree $p$ and total $Y$-degree $q$ is an eigenvector of $\EX\EY$ with eigenvalue $pq$.  Thus a polynomial has no mixed $X,Y$-monomials if and only if it belongs to $\ker(\EX\EY)$.

\begin{prop}\label{prop:dm-to-m}
After the extension in \eqref{eq:dm-plus}, the map in \eqref{eq:dm-to-m} is an injective weight-preserving linear map
\[
 \iota:\mathfrak{dm}_0^+
 \lhook\joinrel\longrightarrow
 \mathfrak d\cap\mathfrak m.
\]
Its image is exactly
\begin{equation}\label{eq:separated-image}
 \left\{A\in\mathfrak B
 \ \middle|\
 A_r=P_r(X_1,\ldots,X_r)+Q_r(Y_1,\ldots,Y_r)
 \text{ for every }r
 \right\}
 =\ker(\EX\EY)\cap\mathfrak B.
\end{equation}
The image of $\mathfrak{dm}_0$ is the part of this space in weights at least three.
\end{prop}

\begin{proof}
We first compare $\iota$ with the embedding $\theta:\mathfrak{dm}_0\to\mathfrak{bm}_0$ of Burmester \cite{Bu1}*{Theorem~4.28} into her balanced double shuffle Lie algebra, and with the bimould realization $\Phi:\mathfrak{bm}_0\to\mathfrak B$ of \cite{BKa}*{Section~8}.  Recall that $\theta(\psi)=\theta_X^{\mathrm{anti}}(\psi)+\theta_Y(\psi_*)$, where $\theta_X^{\mathrm{anti}}$ sends a word $x_{\epsilon_1}\cdots x_{\epsilon_l}$ to $b_{\epsilon_l}\cdots b_{\epsilon_1}$ and $\theta_Y$ sends $y_s$ to $b_s$, and that $\Phi(f)$ is obtained by reading the coefficients of the words $b_{k_1}b_0^{m_1}\cdots b_{k_r}b_0^{m_r}$, $k_i\geq1$, of $f$ as the coefficients of $\prod_iX_i^{k_i-1}Y_i^{m_i}$ and then substituting $Y_1+\cdots+Y_i$ for $Y_i$.  By \cite{Bu1}*{Theorem~4.28} and \cite{BKa}*{Equation~(8.2)}, the map $\Phi\circ\theta$ is an injection of $\mathfrak{dm}_0$ into $\mathfrak B$.  We claim that
\begin{equation}\label{eq:iota-phi-theta}
 \iota(\psi)=(-1)^{k-1}\,\Phi(\theta(\psi))
 \qquad(\psi\in(\mathfrak{dm}_0)_k).
\end{equation}
The words of $\theta_Y(\psi_*)$ contain no $b_0$, so they are read with all $m_i=0$ and give $\sum c_\psi(s_1,\ldots,s_r)\prod_iX_i^{s_i-1}$, together with the constant $\frac{(-1)^{k-1}}{k}a_\psi$ in depth $k$ coming from the correction term of $\psi_*$.  Only the words not beginning with $b_0$ are read, so from $\theta_X^{\mathrm{anti}}(\psi)$ only the words of $\psi$ ending in $x_1$ contribute.  The word $x_0^{s_1-1}x_1\cdots x_0^{s_r-1}x_1$ is sent to $b_1b_0^{s_r-1}\cdots b_1b_0^{s_1-1}$, which is read as $\prod_iY_i^{s_{r+1-i}-1}$, and the substitution gives $\prod_i(Y_1+\cdots+Y_{r+1-i})^{s_i-1}$.  Comparing with \eqref{eq:dm-to-m} proves \eqref{eq:iota-phi-theta}.  Hence $\iota(\psi)\in\mathfrak B$, and $\iota$ is injective on $\mathfrak{dm}_0$.  Since $\iota$ preserves the weight and $\xi_1$ has weight one, it is injective on $\mathfrak{dm}_0^+$.

Next we show that $\xi_1$ commutes with $\iota(\psi)$.  The argument is that of \cite{BKa}*{Proposition~9.1(iii)}.  Put $h=\{\theta(\psi),b_1\}_{\mathrm b}$, the bracket of \cite{BKa}*{Section~8}.  By construction, every word of $\theta(\psi)$ either lies in $\QQ\langle b_0,b_1\rangle$ or contains no $b_0$.  The words of the first kind have zero bracket with $b_1$ by \cite{Bu1}*{Lemma~3.15(i)}, while the derivation defining the bracket preserves the words without $b_0$.  Hence $h$ contains no $b_0$.  By \cite{BKa}*{Theorem~B}, $h\in\mathfrak{bm}_0$, so it is primitive and its projection is invariant under the involution $\tau$ of \cite{BKa}*{Section~8}, which sends every word without $b_0$ other than a power of $b_1$ to a word containing $b_0$.  Thus $h$ is a multiple of $b_1^{k+1}$, which is not primitive, and $h=0$.  Since $\Phi$ intertwines the two brackets by \cite{BKa}*{Lemma~8.2}, we obtain $\uri(\xi_1,\iota(\psi))=0$.  Together with $\iota(\sigma_1)=\xi_1$, this shows that the image of $\iota$ lies in $\mathfrak d$.

Every component in \eqref{eq:dm-to-m} is the sum of a polynomial in the $X$-variables, a polynomial in the $Y$-variables and a constant.  Therefore
\[
 \checkD\bigl(\iota(\psi)\bigr)
 =\sum_i\partial_{X_i}\partial_{Y_i}
 \bigl(\iota(\psi)\bigr)=0,
\]
which proves that $\iota(\psi)\in\mathfrak m$.

We now describe the $Y$-part of $\iota(\psi)$.  For $j\geq1$, put $C_j=\operatorname{ad}(x_0)^{j-1}(x_1)$, and let $A_\psi=\operatorname{ma}(\psi)$ be the mould of \cite{Bu1}*{Definition~B.67}, regarded as a bimould which only depends on the $Y$-variables.  Set
\[
 P_{\psi,r}(Z_1,\ldots,Z_r)
 =\sum_{\substack{s_1+\cdots+s_r=k\\s_i\geq1}}
 c_\psi(s_1,\ldots,s_r)\prod_{i=1}^rZ_i^{s_i-1}.
\]
Writing $\psi^{(r)}$ for the part of $\psi$ of degree $r$ in the alphabet $\{C_1,C_2,\ldots\}$, expansion of the $C_j$ gives
\[
 P_{\psi,r}(Z_1,\ldots,Z_r)
 =\rho_C(\psi^{(r)})(Z_1-Z_2,\ldots,Z_{r-1}-Z_r,Z_r),
\]
where $\rho_C$ is the coefficient map from \cite{Bu1}*{Definition~B.67}.  Since $\psi^{(r)}$ is primitive, its antipode shows that reversing its $C$-words multiplies it by $(-1)^{r-1}$.  Therefore, with $S_i=Y_1+\cdots+Y_i$, we obtain
\[
 P_{\psi,r}(S_r,\ldots,S_1)
 =(-1)^{r-1}\rho_C(\psi^{(r)})(Y_1,\ldots,Y_r)
 =(A_\psi)_r(Y_1,\ldots,Y_r).
\]
The swap of this polynomial is $P_{\psi,r}(X_1,\ldots,X_r)$.  Hence
\begin{equation}\label{eq:iota-ma}
 (-1)^{k-1}\iota(\psi)
 =A_\psi+\swapop(A_\psi)+C_\psi,
\end{equation}
where $C_\psi$ is the constant bimould with the single component $\frac{(-1)^{k-1}}{k}a_\psi$ in depth $k$.  The components of $A_\psi$ have no constant term, since $\psi$ is a Lie polynomial and $(\psi\mid x_1)=0$.

It remains to identify the image among the separated bimoulds.  In weight one, the separated elements of $\mathfrak B_1$ form $\QQ\xi_1$, while there is no such element in $\mathfrak B_2$.  Let $F\in\mathfrak B_k$ be a separated homogeneous element with $k\geq3$. In every depth there is a unique decomposition
\[
 F_r=A_r(Y_1,\ldots,Y_r)+B_r(X_1,\ldots,X_r)+C_r,
\]
where $A_r$ and $B_r$ have zero constant term.  Swap invariance gives $B=\swapop(A)$.  Restricting the alternil coefficient map of $F$ to words in the letters $y_{a,0}$ shows that $\swapop(A)+C$ is alternil.

We next apply alternility to stuffle products of words in the letters $y_{1,m}$.  Suppose first that the sum of the second indices is positive. Every term containing a merger then has a first index greater than one and a positive second index somewhere.  Its coefficient in $F$ vanishes, since it would be a mixed monomial.  The remaining terms give exactly the shuffle relation for $A$.  When the sum of the second indices is zero, the same relation holds because $A$ has no constant term.  Thus $A$ is alternal.  The depth-one parity of $F$ also shows that $A_1$ is even. These properties say that $A$ belongs to the Lie algebra $\ARI^{\mathrm{pol}}_{\underline{\mathrm{al}}*\underline{\mathrm{il}}}$ of \cite{Bu1}*{Theorem~B.45}, whose elements are the alternal polynomial moulds with even depth-one component whose swap is alternil up to a constant mould.  By \cite{Bu1}*{Theorem~B.68}, there is a unique $\psi\in(\mathfrak{dm}_0)_k$ with $A=\operatorname{ma}(\psi)=A_\psi$. The constant mould $C$ is unique with the property that $\swapop(A)+C$ is alternil.  To see this, notice that a constant alternil mould has no component in depth at least two, which follows by applying its coefficient map to the stuffle product of $y_{1,0}$ and $y_{1,0}^{r-1}$, and that no depth-one constant occurs since $k\geq3$. Applying the same restriction argument to $(-1)^{k-1}\iota(\psi)\in\mathfrak B$ shows that $\swapop(A_\psi)+C_\psi$ is alternil, so $C=C_\psi$, and \eqref{eq:iota-ma} gives $F=(-1)^{k-1}\iota(\psi)$.
\end{proof}

The argument at the end of the proof shows more generally that a separated element of $\mathfrak B$ of weight at least two is determined by the parts of its components which depend only on the $Y$-variables and have no constant term.  For the proof of \cref{thm:E} we also need the following two facts.  For a polynomial bimould $F$, let $F|_{X=0}$ be the bimould obtained by setting all $X$-variables equal to zero.

\begin{lem}\label{lem:x-free}
\begin{enumerate}
\item[(i)] Let $F,G\in\BARI^{\mathrm{pol}}$ have no mixed
$X,Y$-monomials.  Then $\uri(F,G)|_{X=0}=\ari(F|_{X=0},G|_{X=0})$.
\item[(ii)] Let $A,B$ be polynomial bimoulds depending only on the
$Y$-variables whose components have no constant term, and let $C,C'$ be bimoulds whose components are constants.  Then $\ari(A+C,B+C')=\ari(A,B)$.
\end{enumerate}
\end{lem}

\begin{proof}
(i) Setting $X=0$ is a ring homomorphism, and all substitutions in the definitions of $\muop$ and $\gaxit$ are homogeneous linear in the $X$-variables.  Hence they commute with setting $X=0$, and the same holds for $\arit$ and $\anit$.  By \eqref{eq:uri-preuri} and \eqref{eq:preuri-definition}, it remains to show $C(G)|_{X=0}=0$.  In \eqref{eq:Lmr} and \eqref{eq:Rmr}, all nodes become zero, and a divided difference of order $r$ with coincident nodes is $(-1)^r$ times the coefficient of $t^r$.  Thus, up to the common factor $(-1)^r$, $L_{m,r}(G)|_{X=0}$ is the coefficient of $t^r$ in $G_{m-r}(-t,\ldots,-t\mid Y_r,\ldots,Y_{m-1})$, and $R_{m,r}(G)|_{X=0}$ is the coefficient of $t^r$ in $G_{m-r}(-t,\ldots,-t\mid Y_{r+1},\ldots,Y_m)$. A monomial of $G_{m-r}$ in the $X$-variables alone contributes the same constant to both, and a monomial in the $Y$-variables alone contributes nothing, since $r\geq1$.  Hence $C(G)|_{X=0}=0$.

(ii) By linearity and antisymmetry it suffices to show that $\ari(A,c)=0$ and $\ari(c,c')=0$ for constants $c$ and $c'$ concentrated in depths $l$ and $l'$.  In $\arit(c)(A)$ the block of $c$ has length $l$, and the terms of amit and of anit both run over all ways of merging $l+1$ consecutive $Y$-variables in an argument of $A$, since $A$ does not depend on the $X$-variables.  Hence $\arit(c)(A)=0$ and $\preari(A,c)_{l+s}=c\,A_s(Y_1,\ldots,Y_s)$.  In $\arit(A)(c)$ the block of $A$ has length $s$, and the $l$ positions of this block relative to the letters of $c$ give
\[
 \arit(A)(c)_{l+s}
 =c\sum_{j=1}^{l}\bigl(A_s(Y_j,\ldots,Y_{j+s-1})-A_s(Y_{j+1},\ldots,Y_{j+s})\bigr)
 =c\bigl(A_s(Y_1,\ldots,Y_s)-A_s(Y_{l+1},\ldots,Y_{l+s})\bigr),
\]
so $\preari(c,A)_{l+s}=c\,A_s(Y_1,\ldots,Y_s)$ as well, and $\ari(A,c)=0$.  For two constants, $\arit(c')(c)=0$ by the same counting, and $\muop(c,c')=\muop(c',c)$, so $\ari(c,c')=0$.
\end{proof}
We now prove that the injection in \cref{prop:dm-to-m} is an isomorphism onto $\mathfrak d$.  The defining conditions of $\mathfrak B$ play no role in the argument.

\begin{thm}\label{thm:linear-kernel}
If $A\in\BARI^{\mathrm{pol}}$ satisfies $\uri(\xi_1,A)=0$, then no component of $A$ contains a mixed $X,Y$-monomial.
\end{thm}

We first reduce the statement to a polynomial in one depth.  For $n\geq1$, let
\[
 R_n=\QQ[X_1,\ldots,X_n,Y_1,\ldots,Y_n],
\]
and let $R_n^{a,b}$ be its subspace of polynomials of total $X$-degree $a$ and total $Y$-degree $b$.  Regarding $p\in R_n$ as a bimould concentrated in depth $n$, put
\begin{equation}\label{eq:Lambda}
 \Lambda_n(p)
 =p^{h,n+1}-p^h+p^t-p^{t,1}+K_n(p),
 \qquad
 K_n(p)=\sum_{i=1}^{n}\bigl(\varphi_{i+1}^-p-\varphi_i^+p\bigr)_{n+1},
\end{equation}
in the notation of \cref{lem:partial-one} with $r=n+1$.  Then $\Lambda_n(p)\in R_{n+1}$.  Every substitution in \eqref{eq:Lambda} is linear in the $X$-variables and linear in the $Y$-variables, so $\Lambda_n$ maps $R_n^{a,b}$ to $R_{n+1}^{a,b}$.

\begin{lem}\label{lem:Lambda-reduction}
\cref{thm:linear-kernel} holds if $\Lambda_n$ is injective on $R_n^{a,b}$ for all $n\geq1$ and all $a,b\geq1$.
\end{lem}

\begin{proof}
Let $\uri(\xi_1,A)=0$.  We show by induction on $r$ that $A_r$ has no mixed monomials.  By \cref{lem:partial-one}, the component $(\partial_1A)_{r+1}$ equals $\Lambda_r(A_r)+q^{t,1}$, where $q=\lambda(A)_r+A_r$ only involves $A_1,\ldots,A_{r-1}$ by \eqref{eq:lambda-recursion}, and $q=0$ for $r=1$.  If $A_1,\ldots,A_{r-1}$ have no mixed monomials, then the recursion \eqref{eq:lambda-recursion} shows that $\lambda(A)_1,\ldots,\lambda(A)_{r-1}$ and $q$ have none either, and the substitution $X_j\mapsto X_j-X_1$ preserves this property.  Hence $\Lambda_r(A_r)$ has no mixed monomials. Since $\Lambda_r$ preserves the bidegree, every component of $A_r$ of bidegree $(a,b)$ with $a,b\geq1$ lies in the kernel of $\Lambda_r$ on $R_r^{a,b}$, and therefore vanishes.
\end{proof}

The injectivity is proved in two steps.  In the first step we show that the kernel of $\Lambda_n$ consists of polynomials which are invariant under simultaneous translation of the $X$-variables, and in the second step we treat these polynomials in difference variables.  Both steps use a lexicographic monomial order and a leading-term argument.  For an exponent vector $\gamma=(\gamma_1,\ldots,\gamma_m)$ and $1\leq j\leq m+1$, write
\[
 \iota_j(\gamma)=(\gamma_1,\ldots,\gamma_{j-1},0,\gamma_j,\ldots,\gamma_m)
\]
for the vector obtained by inserting a zero in position $j$.  Given $\alpha,\beta\in\mathbb Z_{\geq0}^n$, we write $X^\alpha Y^\beta=\prod_iX_i^{\alpha_i}Y_i^{\beta_i}$.

\begin{lem}\label{lem:gamma-injective}
Let $n\geq1$ and put $\Gamma_n(p)=-p^h+p^t+K_n(p)$ for $p\in R_n$.  Then $\Gamma_n$ is injective on the span of the monomials $X^\alpha Y^\beta$ with $\beta\neq0$.
\end{lem}

\begin{proof}
Order the monomials of $R_{n+1}$ lexicographically with respect to
\[
 Y_{n+1}\succ Y_n\succ\cdots\succ Y_1\succ X_{n+1}\succ\cdots\succ X_1.
\]
Let $\beta\neq0$ and let $t$ be the smallest index with $\beta_t>0$.  We claim that the leading monomial of $\Gamma_n(X^\alpha Y^\beta)$ is
\begin{equation}\label{eq:gamma-leading}
 X^{\iota_{t+1}(\alpha)}Y^{\iota_1(\beta)},
\end{equation}
with coefficient one.  The $Y$-exponents occurring in $\Gamma_n(X^\alpha Y^\beta)$ are $(\beta,0)$ from $p^h$, the vector $\iota_1(\beta)$ from $p^t$, and, from the two terms with index $i$ in $K_n$, the vectors obtained from $\iota_i(\beta)$ by moving part of the entry $\beta_i$ from position $i+1$ to position $i$.  For fixed $i$ the largest of the latter is $\iota_i(\beta)$.  Now let $i>t$, or $i=n+1$ with $\iota_{n+1}(\beta)=(\beta,0)$, and let $k<i$ be the largest index with $\beta_k>0$.  Reading the coordinates from position $n+1$ downwards, $\iota_i(\beta)$ agrees with $\iota_1(\beta)$ down to position $k+2$ and has the entry $0$ in position $k+1$, where $\iota_1(\beta)$ has the entry $\beta_k>0$.  Hence $\iota_1(\beta)$ is the largest $Y$-exponent, and it occurs in $p^t$ and in the terms with index $i\leq t$, where it is obtained by putting the whole power $\beta_i$ into position $i+1$, with coefficient one.  The corresponding $X$-exponents are $\iota_1(\alpha)$ for $p^t$, $\iota_i(\alpha)$ for $\varphi_{i+1}^-$ and $\iota_{i+1}(\alpha)$ for $\varphi_i^+$.  Therefore the coefficient of $Y^{\iota_1(\beta)}$ in $\Gamma_n(X^\alpha Y^\beta)$ is
\[
 X^{\iota_1(\alpha)}
 +\sum_{i=1}^{t}\bigl(X^{\iota_{i+1}(\alpha)}-X^{\iota_i(\alpha)}\bigr)
 =X^{\iota_{t+1}(\alpha)},
\]
which proves the claim.  The monomial \eqref{eq:gamma-leading} determines $\beta$ by deleting the first coordinate of its $Y$-exponent, hence $t$, and then $\alpha$ by deleting the coordinate $t+1$ of its $X$-exponent.  Thus distinct monomials have distinct leading monomials, and the lemma follows.
\end{proof}

We call $p\in R_n$ translation invariant if $p(X_1+u,\ldots,X_n+u\mid Y)=p(X\mid Y)$ for an indeterminate $u$. Substituting $u=-X_n$ shows that such a $p$ is a polynomial in the $Y$-variables and in $U_j=X_j-X_{j+1}$, $1\leq j\leq n-1$, and this expression is unique since these are algebraically independent.  For $n=1$, a translation invariant $p$ does not involve $X_1$.

\begin{lem}\label{lem:K-injective}
Let $n\geq2$.  Then $K_n$ is injective on the span of the translation invariant polynomials $U^\alpha Y^\beta$ with $\alpha\in\mathbb Z_{\geq0}^{n-1}$, $\alpha\neq0$, and $\beta\in\mathbb Z_{\geq0}^n$.
\end{lem}

\begin{proof}
Write $U_j=X_j-X_{j+1}$ for $1\leq j\leq n$ in $R_{n+1}$.  For $q=q(U_1,\ldots,U_{n-1}\mid Y_1,\ldots,Y_n)$, the substitutions in $K_n$ become
\[
 K_n(q)=\sum_{i=1}^n\Bigl(
 q\bigl(P_i(U)\mid M_i(Y)\bigr)-q\bigl(Q_i(U)\mid M_i(Y)\bigr)\Bigr),
\]
where $M_i(Y)=(Y_1,\ldots,Y_{i-1},Y_i+Y_{i+1},Y_{i+2},\ldots,Y_{n+1})$ and
\begin{align*}
 P_i(U)&=(U_1,\ldots,U_{i-1},U_i+U_{i+1},U_{i+2},\ldots,U_n)
 \quad(i<n),&
 P_n(U)&=(U_1,\ldots,U_{n-1}),\\
 Q_i(U)&=(U_1,\ldots,U_{i-2},U_{i-1}+U_i,U_{i+1},\ldots,U_n)
 \quad(i>1),&
 Q_1(U)&=(U_2,\ldots,U_n).
\end{align*}
Namely, $\varphi_{i+1}^-$ deletes $X_{i+1}$, so the differences of consecutive remaining $X$-variables are $P_i(U)$, and $\varphi_i^+$ deletes $X_i$, which gives $Q_i(U)$.  Since $U_1,\ldots,U_n$ are linearly independent linear forms in $X_1,\ldots,X_{n+1}$, it suffices to prove injectivity of this expression.

Order the monomials in $U_1,\ldots,U_n,Y_1,\ldots,Y_{n+1}$ lexicographically with respect to
\[
 U_1\succ\cdots\succ U_n\succ Y_{n+1}\succ\cdots\succ Y_1.
\]
Let $\alpha\neq0$ and let $s$ be the largest index with $\alpha_s>0$.  The largest $U$-exponent in the term with $P_i$ is $\iota_{i+1}(\alpha)$, obtained by putting the whole power $\alpha_i$ into $U_i$, and in the term with $Q_i$ it is $\iota_i(\alpha)$.  These vectors equal $(\alpha,0)=\iota_n(\alpha)$ exactly when $i\geq s$, respectively $i>s$, and are strictly smaller otherwise, since at the first position $k>i$ with $\alpha_k>0$ they have the entry $0$.  For $i>s$ the two contributions to the $U$-exponent $(\alpha,0)$ carry the same $Y$-polynomial $Y^\beta$ evaluated at $M_i(Y)$ with opposite signs, so they cancel.  The only remaining contribution comes from $P_s$, and its leading $Y$-monomial is obtained by putting the whole power $\beta_s$ into $Y_{s+1}$.  Hence the leading monomial of $K_n(U^\alpha Y^\beta)$ is $U^{(\alpha,0)}Y^{\iota_s(\beta)}$ with coefficient one.  It determines $\alpha$, hence $s$, and then $\beta$ by deleting the coordinate $s$ of its $Y$-exponent.  The lemma follows.
\end{proof}

\begin{proof}[Proof of \cref{thm:linear-kernel}]
By \cref{lem:Lambda-reduction} it suffices to show that $\Lambda_n(p)=0$ implies $p=0$ for $p\in R_n^{a,b}$ with $a,b\geq1$.  Let $u$ be an indeterminate and write
\[
 p(X_1+u,\ldots,X_n+u\mid Y)=\sum_{m=0}^{a}u^m\,p_m(X\mid Y),
 \qquad p_0=p,
\]
where each $p_m$ is a polynomial all of whose monomials have $Y$-degree $b$.  Substitute $X_j+u$ for $X_j$ in $\Lambda_n(p)$ for $1\leq j\leq n+1$.  The terms $p^{h,n+1}$ and $p^{t,1}$ only involve differences of $X$-variables and are unchanged.  On the other hand, the terms $p^h$, $p^t$ and $K_n(p)$ only delete $X$-variables or keep them, so they become $\sum_mu^m\Gamma_n(p_m)$.  Comparing the coefficients of $u^m$ for $m\geq1$ in $\Lambda_n(p)=0$ gives $\Gamma_n(p_m)=0$, and \cref{lem:gamma-injective} gives $p_m=0$.  Hence $p$ is translation invariant.  For $n=1$ this contradicts $a\geq1$.  When $n\geq2$, substituting $u=-X_{n+1}$, respectively $u=-X_1$, in the identity $p(X+u\mid Y)=p(X\mid Y)$ gives $p^{h,n+1}=p^h$ and $p^{t,1}=p^t$, so $\Lambda_n(p)=K_n(p)=0$.  Since $p$ is homogeneous of $X$-degree $a$ and the $U_j$ are linear forms, $p$ is a combination of monomials $U^\alpha Y^\beta$ with $|\alpha|=a\geq1$.  \cref{lem:K-injective} gives $p=0$.
\end{proof}

\begin{proof}[Proof of \cref{thm:E}]
By \cref{prop:dm-to-m}, the image of $\iota$ is contained in $\mathfrak d\cap\mathfrak m$ and consists of the elements of $\mathfrak B$ without mixed $X,Y$-monomials.  Since $\mathfrak d\subseteq\mathfrak B\subseteq\BARI^{\mathrm{pol}}$, \cref{thm:linear-kernel} shows that every element of $\mathfrak d$ has no mixed monomials, so it lies in the image of $\iota$.  Hence $\iota$ is a linear isomorphism onto $\mathfrak d$, and $\mathfrak d\subseteq\mathfrak m$.

It remains to show that $\iota$ respects the brackets.  Since $\xi_1$ is central in $\mathfrak d$, it suffices to consider homogeneous $\psi,\phi\in\mathfrak{dm}_0$ of weights $k$ and $l$.  By \cref{thm:D}(ii), the bracket $\uri(\iota(\psi),\iota(\phi))$ lies in $\mathfrak d$, so it is a separated element of $\mathfrak B_{k+l}$, as is $\iota(\{\psi,\phi\}_{\mathrm{Ih}})$.  Equations \eqref{eq:iota-ma} and \cref{lem:x-free} give
\[
 \uri(\iota(\psi),\iota(\phi))|_{X=0}
 =(-1)^{k+l}\ari(A_\psi+C_\psi,A_\phi+C_\phi)
 =(-1)^{k+l}\ari(A_\psi,A_\phi).
\]
By \cite{Bu1}*{Theorem~B.68}, the map $\operatorname{ma}$ is a Lie algebra isomorphism from $(\mathfrak{dm}_0,\{\ ,\ \}_{\mathrm{Ih}})$ onto $\ARI^{\mathrm{pol}}_{\underline{\mathrm{al}}*\underline{\mathrm{il}}}$ with the ari bracket of \cite{Bu1}, which is the negative of ours (see the footnote in \cite{BKa}*{Section~8}).  Hence $\ari(A_\psi,A_\phi)=-A_{\{\psi,\phi\}_{\mathrm{Ih}}}$, and the right-hand side above equals $(-1)^{k+l-1}A_{\{\psi,\phi\}_{\mathrm{Ih}}}$, which by \eqref{eq:iota-ma} is the part of $\iota(\{\psi,\phi\}_{\mathrm{Ih}})|_{X=0}$ without constant terms. Since a separated element of $\mathfrak B$ of weight $k+l\geq6$ is determined by this part, as noted after \cref{prop:dm-to-m}, we obtain $\uri(\iota(\psi),\iota(\phi))=\iota(\{\psi,\phi\}_{\mathrm{Ih}})$.
\end{proof}

\section{The higher odd kernels and generators}

The derivations $\partial_{2r+1}$ cannot be injective on $\mathfrak B$, since they annihilate their defining elements $\checkdelta^r(\xi_1)$. We now show that these are the only elements in their kernels when $r\geq1$.  The proof only uses the lowest nonzero depth components. For a bimould $A\neq0$, let $d(A)$ be the smallest depth with $A_{d(A)}\neq0$.

\begin{lem}\label{lem:lowest-component}
\begin{enumerate}
\item[(i)] If $A\in\BARI^{\mathrm{pol}}$ is nonzero and alternil, then
$A_{d(A)}$ is alternal.
\item[(ii)] For nonzero $A,B\in\BARI^{\mathrm{pol}}$ with $d=d(A)$ and
$e=d(B)$, the components of $\uri(A,B)$ in depths below $d+e$ vanish, and
\[
 \uri(A,B)_{d+e}=\ari(A_d,B_e)_{d+e},
\]
where $A_d$ and $B_e$ are regarded as bimoulds concentrated in depths $d$ and $e$.
\item[(iii)] For $r\geq1$, let $f_r$ be the bimould concentrated in depth
one with $(f_r)_1=(X_1Y_1)^r$, and let $p$ be a polynomial regarded as a bimould concentrated in depth $d\geq1$.  With $S=Y_1+\cdots+Y_{d+1}$ and the notation of \cref{lem:partial-one} with $r$ replaced by $d+1$,
\begin{equation}\label{eq:ari-f-r}
\begin{aligned}
 \ari(f_r,p)_{d+1}
 ={}&X_{d+1}^rS^r\,p^{h,d+1}
   -X_1^rS^r\,p^{t,1}
   +(X_1Y_1)^r\,p^t-(X_{d+1}Y_{d+1})^r\,p^h\\
 &-\sum_{i=1}^{d}\Bigl(
  Y_i^r(X_i-X_{i+1})^r\,\varphi_i^+p
  -Y_{i+1}^r(X_{i+1}-X_i)^r\,\varphi_{i+1}^-p\Bigr)_{d+1}.
\end{aligned}
\end{equation}
\end{enumerate}
\end{lem}

\begin{proof}
In the marked-word expansion of $\ganit_C(A)$ for a normalized $C$, the component of depth $d$ is $A_d$ plus terms involving components of $A$ of smaller depth.  Hence $\ganit_{\poc}$ does not change the lowest nonzero component.  If $A$ is alternil, then $\ganit_{\poc}(A)$ is alternal, and alternality in depth $d$ only involves the component of depth $d$.  This proves (i).  For (ii), we use \eqref{eq:uri-preuri} and \eqref{eq:preuri-definition}.  By \eqref{eq:C-definition}, the component $C(B)_m$ only involves $B_{m-1},B_{m-2},\ldots$, so $C(B)$ vanishes in depths at most $e$.  Since $\muop$ and $\arit$ are additive in the depth, the term $\anit(C(B))(A)$ vanishes in depths at most $d+e$, and the lowest component of $\preuri(A,B)$ is $\preari(A_d,B_e)_{d+e}$.  This proves (ii).  For (iii), the word of $\arit(p)(f_r)$ in depth $d+1$ is one block.  Its marked letter is the last one for amit and the first one for anit, which gives the first two terms of \eqref{eq:ari-f-r}.  The terms of $\arit(f_r)(p)$ have a block of length two, with marked second letter for amit and marked first letter for anit, which gives the sum. Finally, the remaining two terms are $\muop(f_r,p)$ and $-\muop(p,f_r)$.
\end{proof}

\begin{lem}\label{lem:sliding-window}
Let $d\geq2$, $r\geq1$ and $P\in\QQ[Z_1,\ldots,Z_d,Y_1,\ldots,Y_d]$.  If
\begin{equation}\label{eq:sliding-window}
 Y_1^r\,P(Z_2,\ldots,Z_{d+1}\mid Y_2,\ldots,Y_{d+1})
 =Y_{d+1}^r\,P(Z_1,\ldots,Z_d\mid Y_1,\ldots,Y_d),
\end{equation}
then $P=c\,(Y_1\cdots Y_d)^r$ for some $c\in\QQ$.
\end{lem}

\begin{proof}
Write $P=\sum c_{\alpha,\beta}Z^\alpha Y^\beta$.  The monomial $Z^\alpha Y^\beta$ contributes to the left-hand side of \eqref{eq:sliding-window} the monomial with exponents $\bigl((0,\alpha),(r,\beta)\bigr)$ and to the right-hand side the monomial with exponents $\bigl((\alpha,0),(\beta,r)\bigr)$.  Both assignments are injective.  If $c_{\alpha,\beta}\neq0$, the monomial on the left must also occur on the right, so $\alpha_d=0$, $\beta_d=r$, and $c_{\alpha,\beta}=c_{\alpha',\beta'}$ with $\alpha'=(0,\alpha_1,\ldots,\alpha_{d-1})$ and $\beta'=(r,\beta_1,\ldots,\beta_{d-1})$.  Repeating the argument $d$ times gives $\alpha=0$ and $\beta=(r,\ldots,r)$.  Conversely, $(Y_1\cdots Y_d)^r$ satisfies \eqref{eq:sliding-window}.
\end{proof}

\begin{prop}\label{prop:odd-higher-depth}
Let $d\geq2$, $r\geq1$, and let $p$ be an alternal polynomial regarded as a bimould concentrated in depth $d$.  If $\ari(f_r,p)_{d+1}=0$, then $p=0$.
\end{prop}

\begin{proof}
All substitutions in \eqref{eq:ari-f-r} are linear in the $X$-variables and in the $Y$-variables, and every term shifts the $X$-degree and the $Y$-degree by $r$.  The alternality relations preserve the bidegree, so the bihomogeneous components of $p$ are alternal, and we may assume that $p$ has total $X$-degree $a$.  Let $u$ be an indeterminate and write
\[
 p(Z_1+u,\ldots,Z_d+u\mid Y)=\sum_{m=0}^au^m\,p_m(Z\mid Y).
\]
Each $p_m$ is alternal, since the simultaneous translation commutes with the permutations of the variable pairs in the alternality relations.

We claim that $p_m=0$ for $m\geq1$, by descending induction on $m$. Suppose that $p_{m+1}=\cdots=p_a=0$ for some $m\geq1$.  Substitute $Z_j+u$ for $X_j$ in \eqref{eq:ari-f-r} and take the coefficient of $u^{m+r}$.  The first two terms have degree at most $r$ in $u$, since the arguments of $p$ are differences.  In the sum, the factors $(X_i-X_{i+1})^r$ are free of $u$, so the coefficient only involves $p_{m+r}=0$.  For the two remaining terms, the coefficient of $u^{m+r}$ in $(Z_1+u)^rY_1^r\sum_ju^jp_j(Z_2,\ldots,Z_{d+1}\mid Y_2,\ldots,Y_{d+1})$ is $Y_1^rp_m(Z_2,\ldots,Z_{d+1}\mid Y_2,\ldots,Y_{d+1})$ by the induction hypothesis, and similarly for the last term.  Therefore
\[
 Y_1^rp_m(Z_2,\ldots,Z_{d+1}\mid Y_2,\ldots,Y_{d+1})
 =Y_{d+1}^rp_m(Z_1,\ldots,Z_d\mid Y_1,\ldots,Y_d),
\]
and \cref{lem:sliding-window} gives $p_m=c_m(Y_1\cdots Y_d)^r$.  This polynomial is symmetric in the variable pairs, so applying its alternality relation to a word of length one and a word of length $d-1$ gives $d\,c_m(Y_1\cdots Y_d)^r=0$.  Hence $c_m=0$, which proves the claim.

Thus $p$ is invariant under simultaneous translation of the $X$-variables.  Substitute $Z_j+u$ for $X_j$ in \eqref{eq:ari-f-r} once more and take the coefficient of $u^r$.  The sum is now free of $u$, and we obtain
\[
 \bigl(S^r-Y_{d+1}^r\bigr)\,p(Z_1,\ldots,Z_d\mid Y_1,\ldots,Y_d)
 =\bigl(S^r-Y_1^r\bigr)\,p(Z_2,\ldots,Z_{d+1}\mid Y_2,\ldots,Y_{d+1}).
\]
Put $C=Y_2+\cdots+Y_d$, which is nonzero since $d\geq2$.  Over an algebraic closure $\mathbb K$ of $\QQ(Y_2,\ldots,Y_d)$, the polynomial $S^r-Y_{d+1}^r$ factors into the linear forms $Y_1+C+(1-\zeta)Y_{d+1}$ and $S^r-Y_1^r$ into the linear forms $Y_{d+1}+C+(1-\eta)Y_1$, where $\zeta$ and $\eta$ run through the $r$-th roots of unity.  Two such forms can only be proportional with factor one, by comparing the constant term $C$, and then the coefficients of $Y_1$ give $\eta=0$, which is impossible.  Hence the two polynomials are coprime in $\mathbb K[Y_1,Y_{d+1}]$.  Since $S^r-Y_{d+1}^r$ is monic in $Y_1$, it has no factor in $\QQ[Y_2,\ldots,Y_{d+1}]$, so the two polynomials are coprime in $\QQ[Y_1,\ldots,Y_{d+1}]$ and hence in $\QQ[Z,Y]$.  Therefore $S^r-Y_{d+1}^r$ divides $p(Z_2,\ldots,Z_{d+1}\mid Y_2,\ldots,Y_{d+1})$. The latter does not involve $Y_1$, while $S^r-Y_{d+1}^r$ has degree $r$ in $Y_1$.  Therefore $p=0$.
\end{proof}

\begin{lem}\label{lem:depth-two}
For $A,B\in\BARI^{\mathrm{pol}}$, the depth-two component of $\uri(A,B)$ depends only on $A_1$ and $B_1$.  Namely,
\begin{equation}\label{eq:depth-two}
 \uri(A,B)_2=\Phi_2(A_1,B_1)-\Phi_2(B_1,A_1),
\end{equation}
where, for polynomials $f,g$ in one pair of variables,
\[
 \Phi_2(f,g)
 =f\bi{X_2}{Y_1+Y_2}g\bi{X_1-X_2}{Y_1}
 -f\bi{X_1}{Y_1+Y_2}g\bi{X_2-X_1}{Y_2}
 +f\bi{X_1}{Y_1}g\bi{X_2}{Y_2}.
\]
\end{lem}

\begin{proof}
By \eqref{eq:preuri-definition} and \eqref{eq:uri-preuri}, it suffices to compute the depth-two component of $\preuri(A,B)$.  Since $C(B)_1=0$, the term $\anit(C(B))(A)$ vanishes in depth two.  In $\arit(B)(A)_2$ the nontrivial block has length two.  For amit its marked letter is the second one, which gives the first term of $\Phi_2$, and for anit it is the first one, which gives the second term with the sign of $\arit=\amit-\anit$.  The last term is $\muop(A,B)_2$.
\end{proof}

\begin{prop}\label{prop:odd-depth-one}
Let $r\geq1$ and let $g\in\QQ[X,Y]$.  If $\Phi_2(f_r,g)=\Phi_2(g,f_r)$ with $\Phi_2$ as in \cref{lem:depth-two} and $f_r=(XY)^r$, then $g=c\,(XY)^r$ for some $c\in\QQ$.
\end{prop}

\begin{proof}
Put $S=Y_1+Y_2$.  Setting $X_2=0$ in $\Phi_2(f_r,g)-\Phi_2(g,f_r)=0$ and dividing by $X_1^r$ gives
\begin{equation}\label{eq:depth-one-restricted}
 -S^rg(-X_1,Y_2)+Y_1^rg(0,Y_2)-Y_1^rg(0,S)+(-1)^rY_2^rg(X_1,S)=0.
\end{equation}
If we also set $X_1=0$ and $Y_2=0$, we get $Y_1^rg(0,Y_1)=0$, so $g(0,Y)=0$, and \eqref{eq:depth-one-restricted} becomes $S^rg(-X_1,Y_2)=(-1)^rY_2^rg(X_1,S)$.  Since $Y_2$ and $S$ are independent variables, $Y_2^r$ divides $g(-X_1,Y_2)$, so $g(x,y)=y^rq(x,y)$, and then $q(-x,a)=(-1)^rq(x,b)$ for independent variables $a,b$.  Hence $q$ does not depend on its second variable, and $g(x,y)=y^rh(x)$.  Finally, setting $Y_2=0$ in $\Phi_2(f_r,g)-\Phi_2(g,f_r)=0$ and dividing by $Y_1^{2r}$ gives
\[
 X_2^rh(X_1-X_2)=(X_1-X_2)^rh(X_2).
\]
With the independent variables $u=X_2$ and $v=X_1-X_2$ this reads $u^rh(v)=v^rh(u)$, so $h(u)=h(1)u^r$.
\end{proof}

\begin{thm}\label{thm:odd-kernels}
For every $r\geq1$, one has
\[
 \ker(\partial_{2r+1})=\QQ\checkdelta^r\xi_1.
\]
\end{thm}

\begin{proof}
Let $r\geq1$ and put $x_r=\checkdelta^r(\xi_1)$.  Since $(\checkdelta A)_1=X_1Y_1A_1$, the depth-one component of $x_r$ is $(X_1Y_1)^r=(f_r)_1$.  Let $A\in\mathfrak B$ be nonzero with $\uri(x_r,A)=0$, and put $d=d(A)$.  By \cref{lem:lowest-component}(i) and (ii), the polynomial $A_d$ is alternal and $\ari(f_r,A_d)_{d+1}=0$.  If $d\geq2$, then \cref{prop:odd-higher-depth} gives $A_d=0$, a contradiction.  Hence $d=1$.  By \cref{lem:depth-two}, the condition $\ari(f_r,A_1)_2=0$ is the hypothesis of \cref{prop:odd-depth-one}, so $A_1=c\,(X_1Y_1)^r$ for some $c\in\QQ$.  The element $A-cx_r$ of $\mathfrak B$ still commutes with $x_r$ and has vanishing depth-one component.  If it were nonzero, its lowest component would have depth at least two, which we have just excluded.  Hence $A=cx_r$.
\end{proof}

Since $\checkD\checkdelta^r(\xi_1)=r^2\checkdelta^{r-1}(\xi_1)\neq0$, the theorem also shows that the restriction of $\partial_{2r+1}$ to $\mathfrak m$ is injective for every $r\geq1$.

The proof only uses that $A$ and $x_r$ are alternil.  Hence the kernel of $A\mapsto\uri(\checkdelta^r(\xi_1),A)$ is one-dimensional already in the uri Lie algebra of all finite-depth polynomial alternil bimoulds.

For every odd $k\geq3$, let $\sigma_k\in\mathfrak{dm}_0$ be the canonical genus-zero element constructed in \cite{DDDHKSSV}, normalized by $(\sigma_k\mid x_0^{k-1}x_1)=1$.  These elements generate the motivic Lie subalgebra of $\mathfrak{dm}_0$, and it is conjectured that
\[
 \mathfrak{dm}_0
 =\operatorname{Lie}
 \left\langle\sigma_k\ \middle|\ k\geq3\text{ odd}\right\rangle
\]
(see \cite{Bu1}*{Conjecture~B.33}).  Together with the central generator in \eqref{eq:dm-plus}, we define
\[
 \xi_k=\iota(\sigma_k)
 \quad(k\geq1\text{ odd}).
\]
The element $\sigma_1$ is formal and does not belong to Racinet's $\mathfrak{dm}_0$.

For a normalized $\sigma_k$, we have $a_{\sigma_k}=1$, so the constant term in depth $k$ is $1/k$.  As an example,
\[
 \Pi_Y(\sigma_3)=y_3+y_2y_1-2y_1y_2,
\]
and \eqref{eq:dm-to-m} gives
\[
 \xi_3=
 \left(
 X_1^2+Y_1^2,\,
 X_1-2X_2-Y_1+Y_2,\,
 \frac13,0,\ldots
 \right).
\]

The normalization gives
\[
 (\xi_1)_1=1,\qquad
 (\xi_k)_1=X_1^{k-1}+Y_1^{k-1}
 \quad(k\geq3\text{ odd})\,.
\]
Put $U=\langle\xi_k\mid k\geq1\text{ odd}\rangle_{\uri_\bullet}$, and let $L$ be the uri Lie subalgebra generated by the $\checkdelta^r(\xi_k)$.  The span identity in \cref{lem:uri-n-generation} shows that $\sum_{r\geq0}\checkdelta^r(U)$ is closed under uri, so $L$ is contained in this sum.  Conversely, $L$ is $\checkdelta$-stable because $\checkdelta$ is a uri derivation and sends each generator $\checkdelta^r(\xi_k)$ to another generator.  Formula \eqref{eq:uri-n} then gives $U\subseteq L$ by induction on the number of operations, and $\checkdelta$-stability gives $\sum_{r\geq0}\checkdelta^r(U)\subseteq L$.  Hence $L=\bigoplus_{r\geq0}\checkdelta^r(U)$, where directness follows from \cref{thm:B}(iii).  Together with the decomposition of $\mathfrak B$ in \cref{thm:B}(iii), this shows that $U=\mathfrak m$ if and only if $L=\mathfrak B$, which proves the equivalence of the two parts of \cref{conj:uri-n-generation}.

\subsection{Generators and relations}\label{sec:generators-relations}

The dimension conjecture for $\mathcal G^{\!f}$ of \cite{BK} (see \cite{BKa}) together with \eqref{eq:Gf-conjecture} predicts the Hilbert--Poincar\'e series of $\mathcal U(\mathfrak B)$.  As in \cite{BKa}, let
\[
 \mathsf D(X)=\frac1{1-X^2},\qquad
 \mathsf O(X)=\frac{X^3}{1-X^2},\qquad
 \mathsf W(X)=\frac{1+X^{12}}{(1-X^4)(1-X^6)}-1,
\]
where $\mathsf W(X)$ is the Hilbert--Poincar\'e series of the period polynomials.  Then
\begin{equation}\label{eq:HU}
 \prod_{k\geq1}\bigl(1-X^k\bigr)^{-\dim_\QQ\mathfrak B_k}
 \overset{?}{=}
 \frac{1}{1-\mathsf D(X)\bigl(X+\mathsf O(X)\bigr)
          +\mathsf D(X)\mathsf W(X)},
\end{equation}
where the left-hand side is the Hilbert--Poincar\'e series of $\mathcal U(\mathfrak B)$ by the Poincar\'e--Birkhoff--Witt theorem. The generator and relation conjecture recorded in \cite{Bu1}*{Conjecture~1.22(ii)--(iii)} explains the denominator. Namely, $\mathsf D(X)(X+\mathsf O(X))$ counts one generator in each odd weight together with its iterates under $\checkdelta$, and $\mathsf D(X)\mathsf W(X)$ counts one relation for each period polynomial together with its $\checkdelta$-derivatives.  In this section we make both counts explicit.  \cref{conj:uri-n-generation} is the generator part, and the relations are studied below.

Notice that \cref{conj:uri-n-generation} does not say that $\dim\mathfrak m_k=1$ in odd weight.  It says that there is one generator which is indecomposable with respect to the $\uri_n$ in every odd weight.  Every uri bracket has zero depth-one component, since $T$ and $T^{-1}$ act as the identity in depth one and the depth-one component of ari vanishes.  Thus every nontrivial expression obtained by iterating the $\uri_n$ has zero depth-one component.  Since $(\xi_k)_1\neq0$, at least one generator in each odd weight is necessary.

Part (ii) of \cref{conj:uri-n-generation} says that $\mathfrak B$ is generated under uri by the elements $\checkdelta^m(\xi_k)$ with $m\geq0$.  They give the canonical choice of the elements denoted by $\xi_k^m$ in \cite{BKa}*{Section~9}, so part (ii) is the corresponding generator conjecture for this uri realization.  The operations $\uri_n$ with $n\geq1$ are needed.  For example, $\uri(\xi_1,\xi_3)=0$ by the Eisenstein relations of \cite{BKa}*{Section~9}, while $\uri_1(\xi_1,\xi_1)$ is a nonzero element of $\mathfrak m_4$.

More generally, let $V$ be any graded subspace of $\mathfrak m$.  Its $\mathfrak{sl}_2$-span has the direct decomposition
\[
 \bigoplus_{r\geq0}\checkdelta^r(V).
\]
This is the analogue of \cref{thm:B}(iii) for $V$, but it describes only the subspace generated by $V$, not all of $\mathfrak B$.  Take $V=\mathfrak d=\iota(\mathfrak{dm}_0^+)$, where the second equality is \cref{thm:E}.  Equation \eqref{eq:separated-image} identifies $\mathfrak d$ with the subspace of $\mathfrak B$ without mixed $X,Y$-monomials.  However, the displayed depth-two component of $\uri_1(\xi_1,\xi_1)$ contains mixed monomials.  This element of $\mathfrak m$ therefore does not lie in $\mathfrak d$.  Directness in \cref{thm:B}(iii) gives
\[
 \mathfrak m\cap\bigoplus_{r\geq0}\checkdelta^r(\mathfrak d)=\mathfrak d,
\]
so the $\mathfrak{sl}_2$-submodule generated by $\mathfrak d$ is proper. It is also not closed under uri, since
\[
 \uri(\xi_1,\checkdelta(\xi_1))
 =\frac12\uri_1(\xi_1,\xi_1).
\]

The first relations coming from cusp forms appear in weight twelve.  We now explain how period polynomials give elements of $\mathfrak m$ which vanish in low depth.  The construction only uses the depth-one components.

Fix an even integer $k\geq4$.  Let
\[
 W_k=\left\{f\in\QQ[x,y]\text{ homogeneous of degree }k-2
 \ \middle|\
 \begin{aligned}
  &f(x,y)+f(y,-x)=0,\\
  &f(x,y)+f(x-y,x)+f(-y,x-y)=0
\end{aligned}\right\}
\]
be the space of period polynomials, and let $W_k^\pm$ be the eigenspaces of $f(x,y)\mapsto f(-x,y)$.  By the Eichler--Shimura isomorphism, $\dim W_k^+=\dim M_k$ and $\dim W_k^-=\dim S_k$, where $M_k$ is the space of modular forms of weight $k$ and $S_k$ is its subspace of cusp forms.  The Eisenstein period polynomial is $p_{E,k}=x^{k-2}-y^{k-2}\in W_k^+$.  For $f,g\in W_k$ of the same parity, put
\begin{equation}\label{eq:P-fg}
\begin{aligned}
 P_{f,g}\bi{X_1,X_2}{Y_1,Y_2}
 ={}&\frac{1}{(k-2)!^2}
 \Bigl[\bigl(f(\partial_a,\partial_b)g(\partial_c,\partial_d)
 +g(\partial_a,\partial_b)f(\partial_c,\partial_d)\bigr)\\
 &\qquad\cdot
 \bigl(X_1ac+X_2ad-Y_1bd+Y_2bc\bigr)^{k-2}
 \Bigr]_{a=b=c=d=0}.
\end{aligned}
\end{equation}
Evaluating at rank-one matrices gives
\[
 P_{f,g}\bi{\alpha\gamma,\alpha\delta}{-\beta\delta,\beta\gamma}
 =f(\alpha,\beta)g(\gamma,\delta)
  +g(\alpha,\beta)f(\gamma,\delta).
\]
Thus \eqref{eq:P-fg} is the inverse map in \cite{Co}*{Theorem~3.42}, after interchanging the $X$- and $Y$-variables. A similar formula is used in \cite{BKS}. By \cite{Co}*{Theorem~3.54(iii)}, it is killed by
\[
 \Delta=\partial_{X_1}\partial_{Y_1}
       +\partial_{X_2}\partial_{Y_2}
\]
and satisfies the three-term identity
\begin{equation}\label{eq:manin}
 0=P_{f,g}\bi{X_1,X_2}{Y_1,Y_2}
 +P_{f,g}\bi{X_2,X_1-X_2}{Y_1+Y_2,Y_1}
 +P_{f,g}\bi{X_2-X_1,X_1}{Y_2,Y_1+Y_2}.
\end{equation}
If $f,g$ run through the unordered pairs of elements of a basis of $W_k^+$ or of a basis of $W_k^-$, the polynomials $P_{f,g}$ are linearly independent, and there are $(\dim M_k)^2$ of them by \cite{Co}*{Theorem~3.54(iii) and Corollary~3.55}.

For odd $\ell\geq1$ and $m\geq0$, the identity $(\checkdelta A)_1=X_1Y_1A_1$ gives
\[
 g_{\ell,m}=\bigl(\checkdelta^m\xi_\ell\bigr)_1
 =\begin{cases}
   (X_1Y_1)^m\bigl(X_1^{\ell-1}+Y_1^{\ell-1}\bigr),&\ell\geq3,\\
   (X_1Y_1)^m,&\ell=1.
  \end{cases}
\]
Given an even $d$, the polynomials $g_{d-2m+1,m}$ with $0\leq m\leq d/2$ form a basis of the symmetric homogeneous polynomials of degree $d$ in one pair of variables.  By \cite{Co}*{Theorem~3.54(iii)}, the polynomial $P_{f,g}$ is antisymmetric under exchange of the two pairs and invariant under both $X_i\leftrightarrow Y_i$ and $(X_i,Y_i)\mapsto(-X_i,-Y_i)$ for each $i$.  Hence it belongs to the exterior square of the span of the $g_{\ell,m}$ and has a unique expansion
\begin{equation}\label{eq:P-expansion}
 P_{f,g}\bi{X_1,X_2}{Y_1,Y_2}
 =\sum c_{(\ell,m),(j,n)}
 \Bigl(g_{\ell,m}\bi{X_1}{Y_1}g_{j,n}\bi{X_2}{Y_2}
      -g_{\ell,m}\bi{X_2}{Y_2}g_{j,n}\bi{X_1}{Y_1}\Bigr),
\end{equation}
where the sum runs over pairs of indices with $\ell+2m+j+2n=k$ and $(\ell,m)<(j,n)$ in lexicographic order.

For $m\geq1$, the identities
\[
 \partial_X\partial_Yg_{\ell,m}
 =m(\ell+m-1)g_{\ell,m-1},
 \qquad
 \checkD\checkdelta^m(\xi_\ell)
 =m(\ell+m-1)\checkdelta^{m-1}(\xi_\ell)
\]
and the derivation property of $\checkD$ show that the first sum below is killed by $\checkD$, since $\Delta P_{f,g}=0$.  It is therefore equal to the sum of the projections of its terms to $\mathfrak m$.  Applying \cref{prop:uri-n-projection} gives the direct definition
\begin{equation}\label{eq:C-fg}
\begin{aligned}
 \mathcal C_{f,g}
 :={}&\sum c_{(\ell,m),(j,n)}
 \uri\bigl(\checkdelta^m\xi_\ell,\checkdelta^n\xi_j\bigr)\\
 ={}&\sum c_{(\ell,m),(j,n)}
 \frac{(-1)^m}{\binom{2(m+n)+\ell+j-2}{m+n}}\,
 \uri_{m+n}(\xi_\ell,\xi_j)\in\mathfrak m_k.
\end{aligned}
\end{equation}
Similar brackets as in the first line of \eqref{eq:C-fg} and their depth-four components are also studied in \cite{BKS}, with other conventions and other choices of the $\xi_\ell$.  In the second line we use our fixed elements $\xi_\ell$ and write $\mathcal C_{f,g}$ in terms of the operations $\uri_n$.  In both lines every term contains two of the generators, where $\checkdelta^m(\xi_\ell)$ counts as one occurrence of $\xi_\ell$. A similar depth-two and parity argument as in the following proof is also used there.

\begin{proof}[Proof of \cref{thm:G}(i)]
By \cref{lem:depth-two} and the first line of \eqref{eq:C-fg}, the depth-two component of $\mathcal C_{f,g}$ is the right-hand side of \eqref{eq:manin}, hence it is zero.  Every uri bracket has zero depth-one component, so $(\mathcal C_{f,g})_1=0$.  Suppose that $(\mathcal C_{f,g})_3\neq0$.  Then $d(\mathcal C_{f,g})=3$, and \cref{lem:lowest-component}(i) shows that $(\mathcal C_{f,g})_3$ is alternal.  It is also swap invariant, since swap preserves the depth and $\mathcal C_{f,g}$ is swap invariant.  Regard $(\mathcal C_{f,g})_3$ as a bimould concentrated in depth three.  Its depth-one component is zero and hence even.  The proof of \cite{Sc}*{Lemma~2.5.5} also works for bimoulds and gives
\[
 (\mathcal C_{f,g})_3(-X_1,-X_2,-X_3\mid -Y_1,-Y_2,-Y_3)
 =(\mathcal C_{f,g})_3(X_1,X_2,X_3\mid Y_1,Y_2,Y_3)\,.
\]
On the other hand, this polynomial has ordinary degree $k-3$, which is odd.  This contradiction proves the claim.
\end{proof}

We use the following period polynomials in weight twelve:
\begin{align*}
 p_E&=x^{10}-y^{10},\\
 p^+_\Delta
 &=\frac{36}{691}(x^{10}-y^{10})-x^2y^2(x^2-y^2)^3,\\
 p^-_\Delta&=4(xy^9+x^9y)-25(x^3y^7+x^7y^3)+42x^5y^5\,.
\end{align*}

\goodbreak
\begin{prop}\label{prop:period-classes}
In weight twelve, the following identities hold.
\begin{enumerate}
\item[(i)] For the even period polynomials,
\begin{align*}
 \mathcal C_{p_E,p_E}
  ={}&-2\uri_0(\xi_1,\xi_{11})=0,\\
 \mathcal C_{p_E,p^+_\Delta}
  ={}&\uri_0(\xi_3,\xi_9)-3\uri_0(\xi_5,\xi_7),\\
 \mathcal C_{p^+_\Delta,p^+_\Delta}
  ={}&\frac{72}{691}\uri_0(\xi_3,\xi_9)
      -\frac{216}{691}\uri_0(\xi_5,\xi_7)
      -\frac{2}{45}\uri_2(\xi_1,\xi_7)
      +\frac{2}{15}\uri_2(\xi_3,\xi_5)
      -\frac{3}{35}\uri_4(\xi_1,\xi_3).
\end{align*}
\item[(ii)] For the odd period polynomial,
\begin{align*}
 \mathcal C_{p^-_\Delta,p^-_\Delta}
  ={}&7\uri_5(\xi_1,\xi_1)
      +\frac{125}{12}\uri_3(\xi_1,\xi_5)
      +\frac{16}{5}\uri_1(\xi_1,\xi_9)\\
     &-\frac{35}{4}\uri_3(\xi_3,\xi_3)
      -20\uri_1(\xi_3,\xi_7)
      +\frac{84}{5}\uri_1(\xi_5,\xi_5).
\end{align*}
\end{enumerate}
\end{prop}

\begin{proof}
Expand \eqref{eq:P-fg} in the basis \eqref{eq:P-expansion} and apply \eqref{eq:C-fg}, using $\uri_0(\xi_1,\xi_{11})=0$.
\end{proof}

For $(f,g)\in \{(p_E,p^+_\Delta),(p^+_\Delta,p^+_\Delta), (p^-_\Delta,p^-_\Delta)\}$, \cref{thm:G}(i) and an exact depth-four calculation give
\[
 \mathcal C_{f,g}
 =\bigl(0,0,0,(\mathcal C_{f,g})_4,\ldots\bigr),
 \qquad
 (\mathcal C_{f,g})_4\neq0.
\]
Setting some of the variables to zero gives, for example,
\begin{align*}
 (\mathcal C_{p_E,p^+_\Delta})_4
 \bi{X_1,X_2,0,0}{0,0,0,0}
 &=\frac{691}{144}X_1X_2(X_1^2-X_2^2)^3,\\
 (\mathcal C_{p^+_\Delta,p^+_\Delta})_4
 \bi{X_1,0,0,0}{Y_1,0,0,0}
 &=\frac1{10}X_1^2Y_1^2
   (11X_1^4-14X_1^2Y_1^2+11Y_1^4),\\
 (\mathcal C_{p^-_\Delta,p^-_\Delta})_4
 \bi{X_1,0,0,0}{Y_1,0,0,0}
 &=\frac16X_1Y_1(X_1^2+Y_1^2)\\
 &\qquad{}\times
   (156X_1^4+827X_1^2Y_1^2+156Y_1^4).
\end{align*}
So these three elements are not zero, and a pair $(f,g)$ does not give a relation among the $\uri_n(\xi_a,\xi_b)$ in general.  In the two cases $f=g$ below, however, we obtain a relation after adding terms with more generators.

Let $\mathcal F_{12}$ be the span of the weight-twelve expressions obtained by iterating the operations $\uri_n$ on the $\xi_k$ with at least four generator occurrences.

\begin{prop}\label{prop:cusp-relation-12}
There are explicit higher correction terms $Q^+$ and $Q^-$ in $\mathcal F_{12}$, given in the \hyperref[sec:weight-twelve-corrections]{Appendix}, such that
\begin{equation}\label{eq:cusp-relation-even}
 \mathcal C_{p^+_\Delta,p^+_\Delta}+Q^+=0,
 \qquad
 \mathcal C_{p^-_\Delta,p^-_\Delta}+Q^-=0.
\end{equation}

In particular,
\begin{equation}\label{eq:cusp-relation-classes}
 \mathcal C_{p^+_\Delta,p^+_\Delta}\equiv0,
 \qquad
 \mathcal C_{p^-_\Delta,p^-_\Delta}\equiv0
 \pmod{\mathcal F_{12}}.
\end{equation}
\end{prop}

\begin{proof}
Both identities were found and verified by an exact computation in SageMath.  Two relations among uri brackets of the elements $\checkdelta^m(\xi_k)$ in weight twelve were first determined, and the Jacobi identity was used to write them in terms of the elements $\checkdelta^m(\xi_k)$ alone.  Applying $\operatorname{pr}_{\mathfrak m}$ and \cref{prop:uri-n-projection} recursively expresses them through iterated $\uri_n$-expressions.  We normalize the two relations so that their terms with two generators are $\mathcal C_{p^+_\Delta,p^+_\Delta}$ and $\mathcal C_{p^-_\Delta,p^-_\Delta}$, respectively.  Jacobi identities and the Eisenstein relations $\uri_0(\xi_1,\xi_k)=0$ reduce the remaining part to $9$ terms involving four generators, $5$ involving six generators and $3$ involving eight generators.  This gives \eqref{eq:cusp-relation-even}.  The second relation initially gives $43$ terms involving at least four generators. With the same reductions this shortens to $10$ terms involving four generators, $8$ involving six generators, $1$ involving eight generators and $1$ involving ten generators.  Every term contains $\xi_1$.  Both identities were then verified exactly in all possible positive depths $1\leq r\leq12$.  The correction terms are given in the \hyperref[sec:weight-twelve-corrections]{Appendix}. This proves \cref{thm:G}(ii).
\end{proof}

The full picture is not yet understood.  In particular, we do not know how to construct correction terms as in \cref{prop:cusp-relation-12} for suitable elements $\mathcal C_{f,g}$ in general and thereby obtain the explicit cusp-form relations predicted by the dimension conjecture \eqref{eq:HU}.

\appendix
\raggedbottom

\section*{Appendix}
\label{sec:weight-twelve-corrections}

We give here the correction terms in \cref{prop:cusp-relation-12}.  To shorten the formulas, recall that $\partial_1=\operatorname{ad}_{\uri}(\xi_1)=\uri_0(\xi_1,-)$. Thus $\partial_1^r$ denotes the $r$-fold iteration of $\uri_0(\xi_1,-)$, and each application adds one occurrence of $\xi_1$. In the formulas below, $Q_d^\pm$ denotes the sum of the terms containing exactly $d$ of the generators $\xi_k$.

For the even relation, write $Q^+=Q^+_4+Q^+_6+Q^+_8$.  The three summands are
\begin{align*}
 Q^+_4={}&
 \frac{17}{36}\uri_0\bigl(\uri_1(\xi_1,\xi_1),
                             \uri_1(\xi_1,\xi_5)\bigr)
 -\frac{4}{45}\uri_0\bigl(\uri_1(\xi_1,\xi_1),
                         \uri_1(\xi_3,\xi_3)\bigr)\\
 &+\frac{13}{30}\uri_0\bigl(\uri_1(\xi_1,\xi_1),
                              \uri_3(\xi_1,\xi_1)\bigr)
 +\frac{59}{448}\partial_1^2\bigl(\uri_2(\xi_1,\xi_5)\bigr)
 -\frac{21}{160}\partial_1
     \bigl(\uri_0(\xi_3,\uri_2(\xi_1,\xi_3))\bigr)\\
 &-\frac{1}{576}\partial_1
     \bigl(\uri_1(\xi_1,\uri_1(\xi_1,\xi_5))\bigr)
 +\frac{47}{2880}\partial_1
     \bigl(\uri_1(\xi_1,\uri_1(\xi_3,\xi_3))\bigr)\\
 &+\frac{7}{240}\partial_1
     \bigl(\uri_1(\xi_1,\uri_3(\xi_1,\xi_1))\bigr)
 -\frac{1}{12}\partial_1
     \bigl(\uri_3(\xi_1,\uri_1(\xi_1,\xi_1))\bigr),
\end{align*}

\begin{align*}
 Q^+_6={}&
 -\frac{34027}{28800}\uri_0
   \bigl(\partial_1^2(\uri_1(\xi_1,\xi_1)),
         \uri_1(\xi_1,\xi_3)\bigr)
 +\frac{28379}{28800}\uri_0
   \bigl(\partial_1^2(\uri_1(\xi_1,\xi_3)),
         \uri_1(\xi_1,\xi_1)\bigr)\\
 &-\frac{26111}{14400}\uri_0
   \bigl(\partial_1(\uri_1(\xi_1,\xi_1)),
         \partial_1(\uri_1(\xi_1,\xi_3))\bigr)\\
 &-\frac{247}{1600}\partial_1^4\bigl(\uri_2(\xi_1,\xi_3)\bigr)
 -\frac{1849}{57600}\partial_1^3
   \bigl(\uri_1(\xi_1,\uri_1(\xi_1,\xi_3))\bigr),
\end{align*}

\begin{align*}
 Q^+_8={}&
 \frac{133319}{483840}\uri_0
   \bigl(\partial_1^4(\uri_1(\xi_1,\xi_1)),
         \uri_1(\xi_1,\xi_1)\bigr)\\
 &+\frac{132509}{241920}\uri_0
   \bigl(\partial_1^3(\uri_1(\xi_1,\xi_1)),
         \partial_1(\uri_1(\xi_1,\xi_1))\bigr)
 -\frac{4489}{241920}\partial_1^5
   \bigl(\uri_1(\xi_1,\uri_1(\xi_1,\xi_1))\bigr).
\end{align*}

For the odd relation, write $Q^-=Q^-_4+Q^-_6+Q^-_8+Q^-_{10}$.  Then
\begin{align*}
 Q^-_4={}&
 -\frac{5347}{540}\partial_1^2\bigl(\uri_1(\xi_1,\xi_7)\bigr)
 +\frac{1051}{120}\partial_1^2\bigl(\uri_1(\xi_3,\xi_5)\bigr)
 -\frac{34099}{1260}\partial_1^2\bigl(\uri_3(\xi_1,\xi_3)\bigr)\\
 &
 -\frac{455}{162}\partial_1
   \bigl(\uri_0(\xi_3,\uri_1(\xi_3,\xi_3))\bigr)
 -\frac{11153}{1944}\partial_1
   \bigl(\uri_1(\xi_1,\uri_0(\xi_3,\xi_5))\bigr)
 +\frac{493}{36}\partial_1
   \bigl(\uri_1(\xi_1,\uri_2(\xi_1,\xi_3))\bigr)\\
 &+\frac{27695}{1008}\partial_1
   \bigl(\uri_2(\xi_1,\uri_1(\xi_1,\xi_3))\bigr)
 -\frac{805}{36}\uri_1
   \bigl(\xi_1,\partial_1(\uri_2(\xi_1,\xi_3))\bigr)\\
 &+\frac{25}{144}\uri_1
   \bigl(\xi_1,\uri_1(\xi_1,\uri_1(\xi_1,\xi_3))\bigr)
 -\frac{475}{48}\uri_2
   \bigl(\xi_1,\partial_1(\uri_1(\xi_1,\xi_3))\bigr),
\end{align*}

\begin{align*}
 Q^-_6={}&
 \frac{138881}{9720}\partial_1^4\bigl(\uri_1(\xi_1,\xi_5)\bigr)
 -\frac{54851}{9720}\partial_1^4\bigl(\uri_1(\xi_3,\xi_3)\bigr)\\
 &-\frac{1045969}{68040}\partial_1^4\bigl(\uri_3(\xi_1,\xi_1)\bigr)
 +\frac{175583}{9072}\partial_1^2
   \bigl(\uri_1(\xi_1,
     \uri_1(\xi_1,\uri_1(\xi_1,\xi_1)))\bigr)\\
 &+\frac{82543}{3888}\partial_1^2
   \bigl(\uri_2(\xi_1,
     \partial_1(\uri_1(\xi_1,\xi_1)))\bigr)
 -\frac{2425}{108}\partial_1
   \bigl(\uri_1(\xi_1,
     \partial_1(\uri_1(\xi_1,\uri_1(\xi_1,\xi_1))))\bigr)\\
 &+\frac{52325}{7776}\uri_1
   \bigl(\xi_1,
     \partial_1^2(\uri_1(\xi_1,\uri_1(\xi_1,\xi_1)))\bigr)
 -\frac{28175}{7776}\uri_1
   \bigl(\xi_1,\uri_1(\xi_1,
     \partial_1^2(\uri_1(\xi_1,\xi_1)))\bigr),
\end{align*}

\begin{align*}
 Q^-_8={}&-\frac{8954621}{1360800}\partial_1^6
   \bigl(\uri_1(\xi_1,\xi_3)\bigr),\qquad
 Q^-_{10}=\frac{5991473}{8164800}\partial_1^8
   \bigl(\uri_1(\xi_1,\xi_1)\bigr).
\end{align*}

\medskip
\noindent\textbf{Use of AI tools.}
ChatGPT~5.6 and Claude Fable~5 were used for language editing, for implementing the operators in SageMath, and for developing the strategy of the swap-invariance proof.  Exact computations with this implementation produced the correction terms in the Appendix.  The author checked all arguments and computations and takes full responsibility for the paper.

\begingroup
\raggedright

\endgroup
\end{document}